\documentclass{article}
\usepackage{graphicx} 
\usepackage{amsmath,amsfonts,amsthm}
\usepackage{hyperref,cleveref,stackengine}
\usepackage{fullpage}
\usepackage{tikz}
\usetikzlibrary{arrows.meta,bending}
\usepackage{algorithm}
\usepackage{algpseudocode}
\usepackage{accents}

\hypersetup{
	colorlinks,
	linkcolor={red!50!black},
	citecolor={blue!50!black},
	urlcolor={blue!80!black}
}

\DeclareMathOperator{\dive}{\mathrm{div}}
\DeclareMathOperator{\du}{d\!}

\DeclareMathOperator{\G}{\mathcal{Q}}
\DeclareMathOperator{\F}{\mathcal{F}}
\DeclareMathOperator{\Om}{\mathcal{O}}
\def\bG{\mathbf{G}}
\def\hf{\widehat{f}}
\def\hu{\widehat{u}}
\DeclareMathOperator{\circphi}{\accentset{\circ}{\varphi}}

\newtheorem{proposition}{Proposition}
\newtheorem{lemma}{Lemma}
\newtheorem{remark}{Remark}
\newtheorem{assumption}{Assumption}
\newtheorem{theorem}{Theorem}
\newtheorem{corollary}{Corollary}
\newtheorem{definition}{Definition}

\title{Low-regret shape optimization in the presence of\\ missing Dirichlet data}
\author {Karl Kunisch \and John Sebastian H. Simon}
\date{ }

\begin{document}
	
	\maketitle
	
	\begin{abstract}
		A shape optimization problem subject to an elliptic equation in the presence of missing data on the Dirichlet boundary condition is considered. It  is formulated by optimizing the deformation field that varies the spatial  domain where the Poisson equation is posed. To take into consideration the missing boundary data the problem is formulated as a no-regret problem
		and approximated by low-regret problems. This approach allows  to obtain deformation fields that are robust against the missing information. The formulation of the regret problems was achieved by employing the Fenchel transform. Convergence of the solutions of the low-regret to the no-regret problems is analysed,
		the gradient of the cost is characterized and a first order numerical method is proposed.
		Numerical examples illustrate  the robustness of the low-regret deformation fields with respect to  missing data.  To the best of our knowledge, this is the first time that a numerical investigation is reported on the level of effectiveness of the low-regret approach in the presence of missing data in an optimal control problem.
		
	\end{abstract}

	\section{Introduction}
	
	{In this work, we seek to advance the understanding of optimization problems involving incomplete or missing data through the lens of low-regret and no-regret formulations. These approaches are grounded in optimization and control theory. The no-regret formulation aims to find solutions that minimize the worst-case discrepancy between the actual and the ideal outcomes, ensuring robustness against all admissible realizations of the missing data. However, this approach can lead to overly conservative solutions, particularly in high-dimensional or ill-posed settings. To address this, the low-regret formulation is introduced as a regularized variant of the no-regret problem. By incorporating a regularization term, it relaxes the strict worst-case criterion, allowing for a controlled trade-off between robustness and performance, and often leading to more stable and computationally tractable solutions.}
	
	{A particular focus of our study lies in evaluating the numerical performance of these formulations, as, to the best of our knowledge, there has been a lack of reported implementations or comprehensive computational studies in the existing literature. To this end, we conduct a detailed investigation using a tracking-type shape optimization problem characterized by incomplete boundary Dirichlet data. This setting provides a meaningful and challenging testbed for assessing the practical effectiveness of the regret-based methodologies.}
	To deal with this problem we utilize the  notion of the no-regret and low-regret control problems as proposed  in \cite{lions1992}. As remarked by J.L. Lions, the notion of regret was first introduced in \cite{savage1972}. We also mention that regret problems can be seen as a way to \textit{increase robustness} of the control with respect to perturbations of certain data.

	Concerning shape optimization problems there are different techniques to quantify domain variations, including the use of level sets and/or considering an arbitrary class of domains satisfying specified properties. In our case, we opt for considering deformation fields that vary a given fixed domain. We will start with a domain $\Omega\subset\mathbb{R}^d$ and find a vector field $\varphi : \mathbb{R}^d\to \mathbb{R}^d$ so that the new domain $\Omega(\varphi) = (I + \varphi)(\Omega)$ minimizes the functional of interest. This quantification is convenient for specifying  a reference system which is vital in the formulation of no and low-regret problems. After appropriate reformulations of the no- and low-regret problems, the gradient of the associated cost-functional of the latter will be characterized  and  a gradient descent algorithm, see e.g., \cite{dapogny2018,murai2013} and the references therein, with a Barzilai-Borwein line search will be described.  

	The specific shape optimization problem that we shall consider is motivated by inverse problems from electrical impedance tomography \cite{roche1996}, where the unknown quantity is the shape. In our case also part of the boundary condition depends on an unknown quantity. In a classical `inverse approach', one might opt for identifying both quantities, the shape and the missing boundary data simultaneously or alternatingly, {see e.g., \cite{habbal2019,boukraa2021,caubet2019,caubet2019b,caubet2020}}. In the no-regret approach, one proceeds by formulating a saddle point problem, in such a way that the classical inverse problem formulation for the unknown shape is still present. The influence of the unknown boundary data is taken into consideration, but the boundary data themselves are not determined.

	To specify the shape optimization problem we introduce a bounded hold-all domain $D\subset\mathbb{R}^d$, for $d \in\{2,3 \}$, with boundary $\Sigma := \partial D$, a {nontrivial} simply connected subdomain $S\subset D$, with boundary $\Gamma:= \partial S$ and the fixed annular subregion $\Omega = D\backslash \overline{S}$. 
	The set of deformation fields is given by
	{\begin{align*}
		\Om(\Omega) := \{ \varphi \in W^{3,\infty}(D;\mathbb{R}^d) : \|\varphi\|_{W^{3,\infty}} \le c ,\,  \overline{[I + \varphi](\Omega)}\subset D\text{, and }\varphi = 0\text{ and } D\varphi= \mathbf{0}\text{ on }\Sigma \},
	\end{align*}
	where $I$ is the identity map on $\mathbb{R}^d$, $0<c<1$, and $\mathbf{0}$ is the $d$-dimensional matrix with entries equal to zero. 
The condition $ D\varphi= \mathbf{0}\text{ on }\Sigma$ in the specification of $\Om(\Omega)$ will be used in the sensitivity analysis of the cost functional in Section \ref{section:derivative}.} 

	Using the notations $\Omega(\varphi):=(I + \varphi)(\Omega)$ and $\Gamma(\varphi) := (I + \varphi)(\Gamma)$, we study the following optimization problem
	\begin{align}
		&\min_{\varphi\in\Om(\Omega)} J(\varphi,g_\delta) := \frac{1}{2}\int_{\omega}
		\left| u - u_d \right|^2 \du x\label{objcon}\\
		\text{subject to}&\nonumber\\
		&\left.
		\begin{aligned}
			-\Delta u & = f  \text{ in } \Omega(\varphi),\quad
			u  = 0 \text{ on }\Gamma(\varphi),\quad
			u = g_r + g_\delta \text{ on }\Sigma,
		\end{aligned}
		\right.\label{poisson}
	\end{align}
	where $\omega \subset D$ is a domain satisfying $\bar \omega  \subset \Omega(\varphi)$ for each $\varphi\in\Om(\Omega)$. Further
$f:D\to \mathbb{R}$  and $u_d:\omega\to\mathbb{R}$ is a given target profile. {Finally the source function on the boundary $\Sigma$ consists of a reference source $g_r:\Sigma \to \mathbb{R}$ and of 
  $g_\delta: \Sigma \to \mathbb{R}$, which  represents  the missing part of the Dirichlet datum.} See \Cref{figure:setup} for an illustration of the set up.
	
	\begin{figure}[h]
		\centering
		\includegraphics[width=7cm]{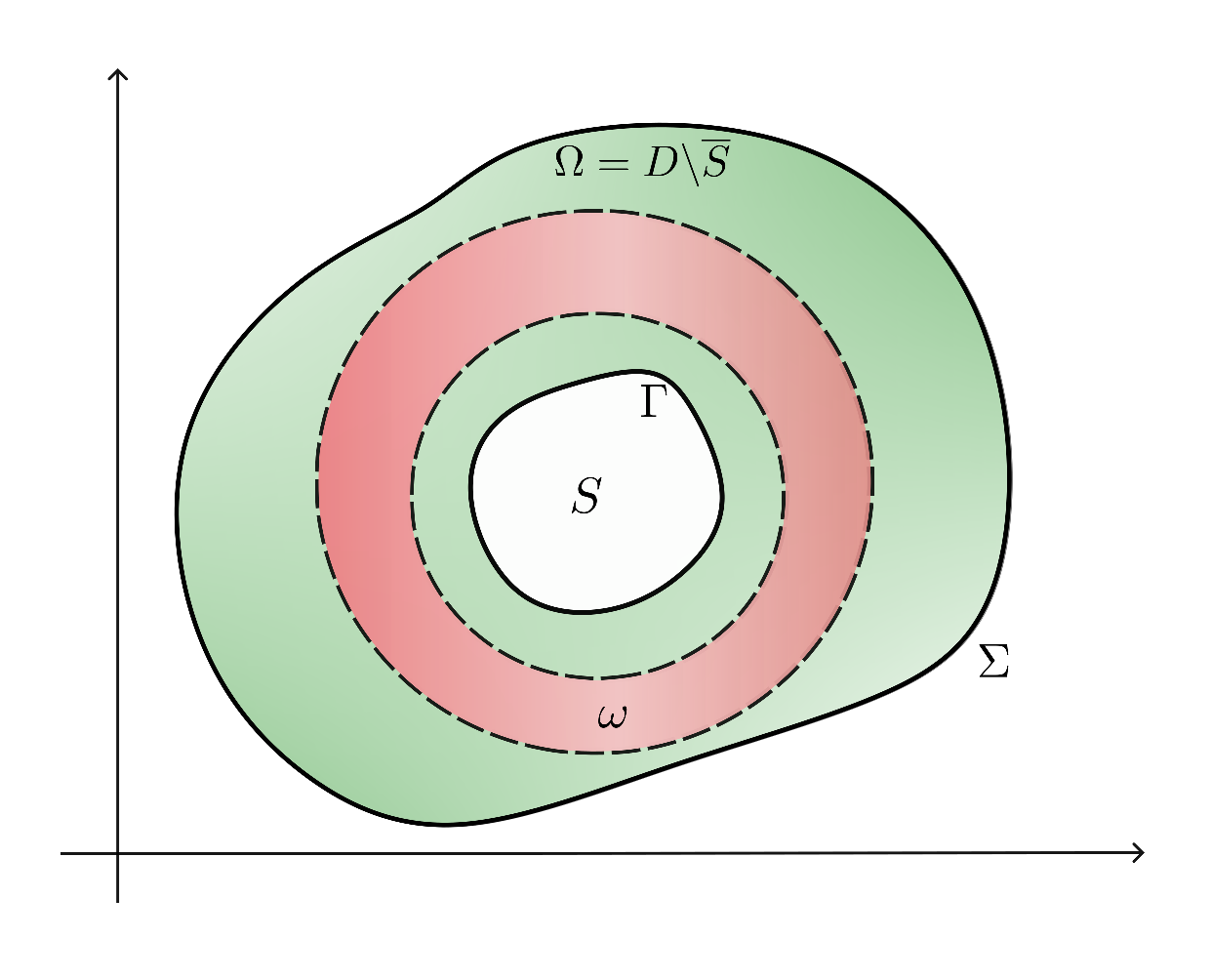}
		\caption{An illustration of the domain $\Omega$, the boundaries $\Gamma$ and $\Sigma$, and the subdomains $\omega$ and $S$.}
		\label{figure:setup}
	\end{figure}
	
{	As the goal is to determine $\varphi\in\Om(\Omega)$ that minimizes the objective functional $J$ for a class of $g_\delta$, in a first step we consider the saddle point problem
	\begin{align*}
		\min_{\varphi\in\Om(\Omega)}\sup_{g_\delta\in \G} J(\varphi,g_\delta),
	\end{align*}
	where $\G:= \{ g\in H^{\frac{1}{2}}(\Sigma) : g_a\le g \le g_b \text{ a.e. on }\Sigma \}$ for some given $g_a,g_b\in H^{\frac{1}{2}}(\Sigma)$ with $g_a(x) \le 0 \le  g_b(x)$, for a.e. $x\in \Sigma$. Incorporating the goal of not wanting to deteriorate too much from a reference  deformation field $\circphi\in \Om(\Omega)$ suggests to replace the previous saddle point problem by:
	\begin{align*}
		\min_{\varphi\in\Om(\Omega)}\sup_{g_\delta\in\G} [J(\varphi,g_\delta) - J(\circphi,g_\delta)].
	\end{align*}
	The problem above is analyzed in  \cite{gabay1994, lions1992}---from which the term \textit{no-regret} problem was introduced---for the treatment of linear optimal control problems with missing initial or boundary data. In the aforementioned papers, the variable $\varphi$ appears as an affine term on the right hand side of the state equation. However, in our situation the state variable $y$ depends on $\varphi$ in a nonlinear fashion.  In such a case, a linearisation of the cost functional $J$ at a reference value for $g$ is carried out  to arrive at a convenient problem formulation for analysis and numerical treatment of low-regret problems, \cite{lions2010,kenne2020}. Considering $g_r$ as an initial guess for the unknown boundary datum suggests to carry out this expansion at
	$0$ and to introduce  $J_1(\varphi,g_\delta):= J(\varphi,0) + \left. \frac{\partial}{\partial g}J(\varphi,g)\right|_{g = 0}g_\delta$. This leads to  the no-regret formulation of  the shape optimization problem with missing boundary data as
	\begin{align}\label{eqk01}
		\min_{\varphi\in\Om(\Omega)}\sup_{g_\delta\in\G} [J_1(\varphi,g_\delta) - J_1(\circphi,g_\delta)],
	\end{align}
	as well as the low-regret formulation 
	\begin{align}\label{eqk1}
		\min_{\varphi\in\Om(\Omega)}\sup_{g_\delta\in \G}\, \left[J_1(\varphi,g_\delta) - J_1(\circphi,g_\delta) - \frac{\varepsilon}{2}\|g_\delta\|_{L^2(\Sigma)}^2\right],
	\end{align}
	where $\varepsilon >0$. The reformulation will  allow a convenient decomposition
	of the cost functionals so that the inner sup-operation can be easily solved. This will be demonstrated in the following subsection. 
} 


	\subsection{Reformulation via decomposition}
{Due to the $\min \sup$ - operation problem \eqref{eqk01} and \eqref{eqk1} represent a significant challenge for numerical approaches. We therefore present a reformulation which will be amenable for practical realizations.  For this purpose it is convenient to introduce the \textit{unknown data-to-state}   operator  $\mathcal{M}_\varphi: H^{\frac{1}{2}}(\Sigma) \to H^1(\Omega(\varphi))$ which maps  $g_\delta\in H^{\frac{1}{2}}(\Sigma)$ to the solution of the Poisson equation \eqref{poisson}. Here and throughout it is assumed that $\varphi\in\Om(\Omega)$,  $f\in L^2(D)$ and the reference boundary term satisfies $g_r\in H^{\frac{1}{2}}(\Sigma)$.} It can be easily shown (see e.g., \cite{casas2006,vexler2007}) that for each $\varphi \in\Om(\Omega)$ the operator $\mathcal{M}_\varphi$ is G{\^a}teaux differentiable at any $g\in H^{\frac{1}{2}}(\Sigma)$ in any direction $\delta g\in H^{\frac{1}{2}}(\Sigma)$. The derivative, which is denoted by $\mathcal{M}_\varphi'(g) \delta g = v(\varphi)$, is the solution to the equation
	\begin{align*}
		\left\{
		\begin{aligned}
			-\Delta v & = 0 &&\text{in }\Omega(\varphi),\\
			v & = 0 &&\text{on }\Gamma(\varphi),\\
			v & = \delta g &&\text{on }\Sigma.
		\end{aligned}
		\right.
	\end{align*}
{By the chain rule,  and using the above expression for  $\mathcal{M}_\varphi'$  with $(g,\delta g)=(0,g_\delta)$ and the definition of $J_1$ we compute}
	\begin{align*}
		\begin{aligned}
			& J_1(\varphi,g_\delta) - J_1(\circphi,g_\delta)  = J(\varphi,0) - J(\circphi,0) +  \left[\frac{\partial}{\partial g}J(\varphi,g)\big|_{g = 0}\right] g_\delta  -  \left[\frac{\partial}{\partial g}J(\circphi,g)\big|_{g = 0}\right] g_\delta\\
			& = J(\varphi,0) - J(\circphi,0) + \int_{\Omega(\varphi)} (\mathcal{M}_\varphi(0) - u_d)\chi_\omega \mathcal{M}_\varphi'(0)g_\delta  \du x - \int_{\Omega} (\mathcal{M}_{\circphi}(0) - u_d)\chi_\omega
			\frac{\partial}{\partial g}\mathcal{M}_{\circphi}'(0)g_\delta  \du x,
		\end{aligned}
	\end{align*}
	where {for a given set $A$ the function $\chi_A: D\to \{0,1 \}$ is the characteristic function of the set $A$ defined as 
	\begin{align*}
		\chi_\omega(x) = \left\{ 
		\begin{aligned}
			1 &\text{ if }x\in A,\\
			0 &\text{ otherwise. }
		\end{aligned}
		\right.
	\end{align*}}
	Introducing  the \textit{adjoint variable} $w(\varphi) \in H_0^1(\Omega(\varphi))\cap H^2(\Omega(\varphi))$ which solves the equation
	\begin{align}\label{adjoint1}
		\left\{
		\begin{aligned}
			-\Delta w & = (u(\varphi) - u_d)\chi_\omega &&\text{in }\Omega(\varphi),\\
			w & = 0 &&\text{on }\partial\Omega(\varphi),
		\end{aligned}
		\right.
	\end{align}
	where $u(\varphi)= \mathcal{M}_\varphi(0)$,
	the above computation  can be continued as follows
	\begin{align*}
		\begin{aligned}
			& J_1(\varphi,g_\delta) - J_1(\circphi,g_\delta) = J(\varphi,0) - J(\circphi,0) + \int_{\Sigma} [ \partial_\nu w(\circphi) - \partial_\nu w(\varphi) ] g_\delta \du s.
		\end{aligned}
	\end{align*}
	The no-regret problem can thus be rewritten as
	\begin{align}\label{no-regreti}
		\min_{\varphi\in\Om(\Omega)}\sup_{g_\delta\in \G} \left[ J(\varphi,0) - J(\circphi,0) + \int_{\Sigma} [ \partial_\nu w(\circphi) - \partial_\nu w(\varphi) ] g_\delta \du s \right],
	\end{align}
{while the low-regret problem is recast in the form
	\begin{align*}
		\min_{\varphi\in\Om(\Omega)}\sup_{g_\delta\in \G} \left[ J(\varphi,0) - J(\circphi,0) + \int_{\Sigma} [ \partial_\nu w(\circphi) - \partial_\nu w(\varphi) ] g_\delta \du s - \frac{\varepsilon}{2}\|g_\delta\|_{L^2(\Sigma)}^2\right].
	\end{align*}
This can be expressed as}
	\begin{align}
		\min_{\varphi\in\Om(\Omega)} \left[ J(\varphi,0) - J(\circphi,0) + \F_\varepsilon^*(\partial_\nu w(\circphi) - \partial_\nu w(\varphi))  \right],
	\end{align}
	where $\F_\varepsilon^* : L^2(\Sigma)\to \mathbb{R}$ is the Fenchel transform of $[g_\delta\mapsto \frac{\varepsilon}{2}\|g_\delta\|_{L^2(\Sigma)}^2]: L^2(\Sigma)\to \mathbb{R}$ defined by
	\begin{align*}
		\F_\varepsilon^*(y) = \sup_{g_\delta\in\G} \left[  \langle y,g_\delta \rangle_\Sigma - \frac{\varepsilon}{2}\|g_\delta\|^2_{L^2(\Sigma)}\right].
	\end{align*}
	
	\begin{remark}\label{dirichletKKT}
		{\em Due to the constraints on the elements of $\G$ and the concavity of the functional inside the supremum in the definition of the Fenchel transform, one can easily establish the existence of $\overline{g}\in \G$ such that $\F_\varepsilon^*(y) = \langle y,\overline{g} \rangle_\Sigma - \frac{\varepsilon}{2}\|\overline{g}\|^2_{L^2(\Sigma)}$. We further infer from the  {Karush-Kuhn-Tucker} conditions that $\overline{g} = P_{[g_a,g_b]}(\frac{1}{\varepsilon}y)$, where $P_{[g_a,g_b]}: L^2(\Sigma)\to L^2(\Sigma)$ is the projection operator on $[g_a,g_b]$. }
	\end{remark}

	\subsection{Related literature and novelty}
	
	Let us discuss related works that led to the formulation of our current problem.
	Aside from the already mentioned seminal works of J.D. Savage \cite{savage1972} and J.L. Lions \cite{lions1992}, we also refer to the work of D. Gabay and J.L. Lions \cite{gabay1994} extending the concept of regret problems to a system of multiple agents. Later O. Nakoulima et al., \cite{nakoulima2003} showed that the concept of Pareto controls and no-regret controls are synonymous. The authors further provided necessary optimality conditions for both no-regret and low-regret optimal control problems, where they first derived such optimality system for the latter and showed that the solution converges to the solution of the optimality system for the former, which turns out to be a singular optimality system. The approach they used ---which we adopted in this exposition--- for the decomposition is based on the first order expansion of the state solution around the zero control.
	
	Recent works dealing with regret control problems include systems modelling populations dynamics: B. Jacob and A. Omrane \cite{jacob2010} worked on an age-structured model with missing initial data; C. Kenne et al., \cite{kenne2020} investigated an age-structured model with missing birth rate leading to a system which depends nonlinearly on the missing data.
	Systems involving fractional time derivatives, to capture appropriate description of memory and hereditary effects, were also studied: D. Baleanu et al., \cite{baleanu2016} worked on a wave equation with a Riemann-Liouville time derivative with unknown data in the initial condition. Meanwhile, G. Mophou \cite{mophou2017} considered perturbations on the boundary data in a heat equation. Lastly, we also mention works on equations modelling physical phenomena such as a wave equation with unknown Dirichlet data and wave speed  \cite{hafdallah2019}, and a model for a thermoelastic body with missing initial condition \cite{hafdallah2020}.
	
	In the mentioned papers, conditions on the set of admissible controls were imposed in such a manner that the last term in \eqref{no-regreti} becomes zero. We do not use such a condition since the box constraints on the missing data guarantees the existence of a maximizer for the said term. We also mention that imposing constraints on the missing data has been mentioned in \cite{lions2010} , however, to best of our knowledge, the analysis with such constraints has not been done and implemented numerically, regardless of shape optimization.

	We also mention that apparently in the articles dealing with no-regret and low-regret problems for optimal control  their implementation and consequently their behavior in numerical practice,  have not been addressed. Numerically verifying that the low-regret and no-regret formulations  are effective for problems with missing data for shape optimization  is therefore a main innovative step of this research, besides the fact that these techniques were not analyzed for shape optimization before. Formulating  the shape optimization problem by means of minimizing with respect to deformation fields, is convenient for a function space analysis of the problem and the numerical gradient method can be  conveniently  combined with a Barzilai--Borwein line-search method.
	
	{
	Beyond the connections to optimal control and regret-based formulations, the present work is also closely related to inverse problems, particularly those concerned with shape identification under partial or uncertain observations. Indeed, reconstructing the shape of a domain subject to PDE constraints with missing boundary data naturally falls within the scope of inverse problems. A classical inverse approach would attempt to recover both the domain and the unknown boundary data, often leading to ill-posed formulations, see e.g., \cite{caubet2019} and the references therein. In contrast, our no-regret and low-regret formulations offer an alternative by optimizing deformation fields that are robust across all admissible realizations of the missing data, without explicitly reconstructing it. The approach is relevant in settings such as electrical impedance tomography (EIT), where the shape and boundary inputs are both sources of uncertainty. Our work thus serves as a bridge between control theory and inverse problems, demonstrating that no-regret-based shape optimization can serve as a viable and numerically effective framework for addressing PDE-constrained inverse problems with partial data. To the best of our knowledge, this is the first implementation and numerical investigation of low-regret methods in this context.}
	We also comment on its close connection as a special case of  conductivity reconstruction problems.
	Such inverse problems have been analyzed as shape optimization problem by e.g,. \cite{roche1996},  and were extensively studied also in e.g., \cite{afraites2022,afraites2024,cherrat2023,eppler2005}. Of course, our current problem is a much simplified version of the mentioned inverse problems. Nevertheless, we look at this article  as a precursor into further analysing such problems with the regret-based approach.

	{For shape optimization much research has also been directed to shape optimization under uncertainty where the unknown quantity, in our case $g_\delta$ is treated as a stochastic random force with known probability distribution. In this case the solution to the state equation becomes a stochastic variable as well. In \cite{conti2009,dambrine2015}, for example, the resulting  optimization problems minimizing the expectation over the probability space of appropriately defined cost-functionals were investigated. This is a different setting from the one followed here, were we concentrate on a worst case scenario. Worst case parametric and shape optimization problems for stationary elliptic systems were also investigated in e.g \cite{allaire2014} but not in a (low)-regret  type formulation and not for setting represented by \eqref{objcon}-\eqref{poisson}.  Abstract calculus of variation techniques were used in \cite{bellido2017} to investigate worst case scenarios for cost functionals with specific structure, as for instance the Dirichlet energy. }

	The remainder of this article is structured as follows: in \Cref{section:existence}, we provide the existence of solutions to the low-regret and no-regret problems. This encapsulates establishing the existence of solutions to the governing state equations and provides the appropriate topology for the deformation fields. \Cref{section:low-to-no} is dedicated to show convergence of solutions of the low-regret problem to a solution of the no-regret problem using  properties of the Fenchel transform. We show the G{\^a}teaux differentiability of the objective functional of the low-regret problem in \Cref{section:derivative}, where we also provide a first-order necessary condition for the optimal deformation field. Numerical examples are provided in \Cref{section:numerics} where we observe the convergence of the free optimal boundary $\Gamma$  as we take $\varepsilon\to 0$, and compare the solutions of the low-regret problem to the solutions we get when $g_\delta$ is set to a  specific value. Concluding remarks and possible future works are given in \Cref{section:conclusion}.

	\section{Existence Analysis}
	\label{section:existence}

	In this section, we establish the existence of minimizing deformation fields for the low-regret and no-regret problems. 
	Also included in the analysis is showing the continuity of the deformation-to-state map under an appropriate closed topology.
	
	Before we proceed, we note that the reformulation of both problems gives rise {to the adjoint variable }$w$. This implies that the regret problems have two governing equations, namely, given $\varphi\in\Om(\Omega)$ we want to find $(u(\varphi),w(\varphi))$ that solves
	\begin{align}\label{eq6}
		\left\{
		\begin{aligned}
			-\Delta u(\varphi) & = f  \text{ in } \Omega(\varphi),\quad
			u(\varphi)  = 0 \text{ on }\Gamma(\varphi),\quad
			u(\varphi) = g_r \text{ on }\Sigma.\\
			-\Delta w(\varphi) & = (u(\varphi) - u_d)\chi_\omega  \text{ in } \Omega(\varphi),\quad
			w(\varphi)  = 0 \text{ on }\partial\Omega(\varphi).
		\end{aligned}
		\right.
	\end{align} 
	We then write the no-regret problem as
	\begin{align}\label{no-regret}
		\min_{\varphi\in\Om(\Omega)} \sup_{g_\delta \in \G} \left[ \widetilde{J}(\varphi) - \widetilde{J}(\circphi) + \int_\Sigma (\partial_\nu w(\circphi) - \partial_\nu w(\varphi))g_\delta \du s\right] 
		\text{ subject to }\eqref{eq6},
	\end{align}
	where the functional $\widetilde{J}:\Om(\Omega)\to \mathbb{R}$ is defined as
	\begin{align*}
		\widetilde{J}(\varphi) = \int_\omega |u(\varphi) - u_d|^2 \du x,
	\end{align*}
	which  also  satisfies $\widetilde{J}(\varphi) = J(\varphi,0)$. 
\
On the other hand, the low-regret problem is written as
	\begin{align}\label{low-regret}
		\min_{\varphi\in\Om(\Omega)} \left[ \widetilde{J}(\varphi) - \widetilde{J}(\circphi) + \F_\varepsilon^*(\partial_\nu w(\circphi) - \partial_\nu w(\varphi))  \right]
		\text{ subject to }\eqref{eq6}.
	\end{align}

	\subsection{Analysis on the governing equations}
	
	To help us in proving the existence of the solution to the governing equations we will rely on the following assumptions:
	\begin{assumption}\label{assumption1}
		The domain $\Omega$, source function $f$, given boundary data $g_r$ and the desired profile $u_d$ satisfy the following:
		\begin{itemize}
			\item[i.] $\Omega$ is of class $\mathcal{C}^{2,1}$;
			\item[ii.] $f\in L^2(D)$, $g_r\in H^{\frac{3}{2}}(\Sigma)$, and $u_d \in H^1(\omega)$.
		\end{itemize}
	\end{assumption}
	
	To establish the existence of minimizing deformation fields for the optimization problems \eqref{no-regret} and \eqref{low-regret} we must first discuss the existence of {weak solutions} to the \textit{one-way coupled system}
	\begin{align}\label{system:couple}
		\left\{
		\begin{aligned}
			-\Delta \widetilde{u} & = f + \Delta u_{g_r} &&\text{in }\Omega,\\
			-\Delta w & = (\widetilde{u} + u_{g_r} - u_d)\chi_\omega &&\text{in }\Omega,\\
			\widetilde{u} & = w  = 0&&\text{on }\partial\Omega,
		\end{aligned}
		\right.
	\end{align}
	where $\Omega\subset D$ is an arbitrary domain that is of class $\mathcal{C}^{2,1}$ and $u_{g_r}\in H^2(D)$ is a function which satisfies $u_{g_r}|_{\Sigma} = g_r$ whose existence is guaranteed by the surjectivity of the trace operator and satisfies
	\begin{align*}
		\|u_{g_r}\|_{H^2(D)} \le c \|g_r\|_{H^{\frac{3}{2}}(\Sigma)},
	\end{align*}
	for some constant $c>0$. To be specific, we choose $u_{g_r}\in H^2(D)$ which solves 
	\begin{align}
		-\Delta u_{g_r} = 0\text{ in }D,
		\quad u_{g_r} = g_r\text{ on }\Sigma.
	\end{align}
	
	\begin{remark}
		Given a strong solution $\widetilde{u}\in H_0^1(\Omega)\cap H^2(\Omega)$ to \eqref{system:couple}, we see that 
		\begin{align}
			u = \widetilde{u}+u_{g_r}\in H^2(\Omega)
		\end{align}
		solves the first equation in \eqref{eq6}.
	\end{remark}
	
	To aid us in the analysis, we also recall the following uniform Poicar{\'e} inequality, which is due to the monotonicity of the eigenvalues of the Dirichlet problem with respect to domains and the minima of the Rayleigh quotient, {see Proposition 3.1.17 and Corollary 4.7.4 in \cite{henrot2010} or Theorem 2.1 of \cite[page 412]{delfour2011}}.
	\begin{lemma}\label{poincare}
		There exists a constant $c_P>0$ such that for each $\Omega\subset D$
		\begin{align}
			\|u\|_{L^2(\Omega)}^2 \le c_P \|\nabla u\|_{L^2(\Omega)^d}^2,\quad \forall u\in H_0^1(\Omega).
		\end{align}
	\end{lemma}

	We say that $(\widetilde{u},w)\in H_0^1(\Omega)^2$ is a weak solution to \eqref{system:couple} if it solves the variational problem
	\begin{align}\label{weakform}
		\int_\Omega \nabla \widetilde{u}\cdot \nabla \psi_1 \du x + \int_\Omega \nabla w\cdot \nabla \psi_2 \du x = \int_\Omega f\psi_1 \du x - \int_\Omega \nabla u_{g_r}\cdot \nabla \psi_1 \du x + \int_\omega (\widetilde{u} + u_{g_r} -u_d)\psi_2 \du x ,
	\end{align}
	for all $\psi_1,\psi_2\in H_0^1(\Omega)$.
	
	The following proposition gives us the existence of the unique weak solution of \eqref{system:couple}.
	
	\begin{proposition}\label{prop:stateexistence}
		Let $\Omega\subset D$, $f$, $g_r$ and $u_d$ satisfy \Cref{assumption1}. Then there exists a unique weak solution $(\widetilde{u},w)\in H_0^1(\Omega)^2$ to \eqref{weakform} such that
		\begin{align}\label{estimate:sol}
			\|\nabla\widetilde{u}\|_{L^2(\Omega)^d} + \|\nabla w\|_{L^2(\Omega)^d} \le c ( \|f\|_{L^2(D)} +  \|g_r\|_{H^{\frac{3}{2}}(\Sigma)} + \|u_d\|_{L^2(\omega)} ),
		\end{align}
		for some constant $c>0$ independent of $\Omega\subset D$.
	\end{proposition}

	{The independence of the constant $c>0$ in \eqref{estimate:sol} is possible due to \Cref{poincare}.}

	\begin{remark}\label{remark:regularity}
		It can be easily shown that the assumptions on the data can be reduced, e.g. $\Omega\subset D$ is Lipschitz, $f\in H^{-1}(\Omega)$, $g\in H^{\frac{1}{2}}(\Sigma)$  and $u_d\in L^2(\omega)$, to establish the existence of the weak solution to \eqref{system:couple}. In fact, the present assumption gives us more regular solutions, i.e., $\widetilde{u}\in H_0^1(\Omega)\cap H^2(\Omega)$ and $w\in H_0^1(\Omega)\cap H^3(\Omega)$ by the classical elliptic regularity, whose norms -- in their respective spaces -- are bounded above by the right-hand side of \eqref{estimate:sol}. We deemed it necessary to assume such properties now as they are needed for the subsequent analyses (existence and sensitivity analysis) on the shape optimization.
	\end{remark}

	Next, we address the issue of continuity of the deformation-to-solution map, i.e. given a sequence $(\varphi_n)\subset \Om(\Omega)$ converging to $\varphi^*\in\Om(\Omega)$, can we prove some convergence of $(\widetilde{u}(\varphi_n),w(\varphi_n))\in H_0^1(\Omega(\varphi_n))^2$ to $(\widetilde{u}(\varphi^*),w(\varphi^*))\in H_0^1(\Omega(\varphi^*))^2$? Here we used the notation $(\widetilde{u}(\varphi),w(\varphi))\in H^1_0(\Omega(\varphi))$ to denote the weak solution of \eqref{system:couple} in the domain $\Omega(\varphi)$.
	
	The question above highlights two different convergences that we should tackle: 1. with which topology do we endow $\Om(\Omega)$; and 2. what notion of convergence is meant for the convergence of states. We first discuss the issue on the convergence of the deformation fields. The compact embedding $C^{2,1}(\overline{D};\mathbb{R}^d) \hookrightarrow C^{2}(\overline{D};\mathbb{R}^d)$ implies that there exists $\varphi^*\in C^{2}(\overline{D};\mathbb{R}^d)$ such that $\varphi_n\to \varphi^*$ in $C^{2}(\overline{D};\mathbb{R}^d)$. The uniform bound on the elements of $\Om(\Omega)$ implies that $\varphi^*\in \Om(\Omega)$. Having this in mind, when we talk about convergence of sequence in $\Om(\Omega)$, we refer to the topology induced by the space $C^2(\overline{D};\mathbb{R}^d)$ in which the set $\Om(\Omega)$ is closed.
	{
	\begin{definition}
		Let $(\varphi_n) \subset \Om(\Omega)$ be a sequence and $\varphi\in\Om(\Omega)$. We say that the sequence $(\varphi_n)$ converges to $\varphi$ in $\Om(\Omega)$ --- denoted as $\varphi_n \to \varphi$ in $\Om(\Omega)$ --- if and only if the sequence converges with the $C^2(\overline{D};\mathbb{R}^d)$--topology.
	\end{definition}}
	
	To address the second issue, we first notice that since each $\varphi_n\in \Om(\Omega)$ is a diffeomorphism of order $\mathcal{C}^{2,1}$, each $\Omega_n:= \Omega(\varphi_n) \subset D$ is a domain of class $\mathcal{C}^{2,1}$ and satisfies the uniform cone property \cite[Theorem 2.4.7]{henrot2010}. This implies the existence of linear and continuous extension operators $E_n : H_0^1(\Omega_n)^2\to H^1(D)^2$ such that 
	\begin{align}\label{emb}
		\sup_{n\in \mathbb{N}} \|E_n\|_{\mathcal{L}(H_0^1(\Omega_n)^2; H^1(D)^2)} \le K_1,
	\end{align}
	where $K_1>0$ is independent on $\Omega_n$. 
	The extension operator can also be defined such that
	\begin{align}\label{uniext}
		\|E_n(u,w)\|_{\mathcal{L}((H^2(\Omega_n)\cap H_0^1(\Omega_n))^2; H^2(D)^2)} \le K_2,
	\end{align}
with $K_2$ independent of $n$,
{see e.g.  \cite[Theorem II.1]{chenais1975}.}
	Another implication of the uniform cone property is the following compactness of the characteristic functions corresponding to the domains $\Omega(\varphi)$ for $\varphi\in\Om(\Omega)$.
	\begin{lemma}\label{lemma:charcon}
		Let $(\varphi_n) \subset \Om(\Omega)$ be a sequence converging to $\varphi^*\in \Om(\Omega)$. Then a subsequence of the characteristic functions $(\chi_{\Omega_n})\subset L^2(D)$ of $\Omega_n= \Omega(\varphi_n)$ converges strongly to $\chi_{\Omega^*}\in L^2(D)$, which is a.e. the characteristic function corresponding to $\Omega^*:= \Omega(\varphi^*)$. 
	\end{lemma}
	
	\begin{proof}
		From the uniform cone property and \cite[Theorem III.1]{chenais1975}, there exists a subsequence of $(\chi_{\Omega_n})\subset L^2(D)$, which we denote in the same way, and an element $\chi_{\Omega^*}\in L^2(D)$ such that $\chi_{\Omega_n}\to \chi_{\Omega^*}$ in $L^2(D)$. From \cite[Proposition~III.5]{chenais1975}, $\chi_{\Omega^*}$ is a.e. the characteristic function of 
		\begin{align*}
			\Omega^* := \bigcap_{m\in\mathbb{N}} G_m,\text{ where }G_m = \bigcup_{n\ge m}\Omega_n.
		\end{align*}
		One can easily show that $\Omega(\varphi^*) = \Omega^*$, which concludes the proof.
	\end{proof}

	Since {$	\| \chi_{\Omega_n} \|_{L^\infty(D)} \le 1$} the previous lemma also implies that $\chi_{\Omega_n} \to \chi_{\Omega^*}$ in $L^p(D)$, for any $1\le p < + \infty$.
	
	We are now in position to lay down the promised continuity of the deformation-to-state map.
	\begin{proposition}\label{shapecontinuity}
		Suppose that $(\varphi_n) \subset \Om(\Omega)$ is a sequence converging to $\varphi^*\in \Om(\Omega)$. Then there exists $(u^*,w^*)\in H^1(D)^2$ such that $E_n(\widetilde{u}(\varphi_n),w(\varphi_n)) \rightharpoonup (u^*,w^*)$ in $H^1(D)^2$. Furthermore, $(u^*,w^*)_{|\Omega(\varphi^*)} = (\widetilde{u}(\varphi^*),w(\varphi^*))\in H_0^1(\Omega(\varphi^*))^2$.
	\end{proposition}
	\begin{proof}
		From \Cref{prop:stateexistence} {and \eqref{emb}}, we see that
		\begin{align}
			\| E_n(\widetilde{u}(\varphi_n),w(\varphi_n)) \|_{H^1(D)^2} \le c ( \|f\|_{L^2(D)} +  \|g_r\|_{H^{\frac{3}{2}}(\Sigma)} + \|u_d\|_{L^2(\omega)} ),
		\end{align}
		for some $c>0$ independent of $n$. We thus infer the existence of $(u^*,w^*)\in H^1(D)^2$, such that a subsequence of $(E_n(\widetilde{u}(\varphi_n),w(\varphi_n)))\subset H^1(D)^2$, denoted in the same way, satisfies $E_n(\widetilde{u}(\varphi_n),w(\varphi_n)) \rightharpoonup (u^*,w^*)$ in $H^1(D)^2$. 
		
		The compact embedding $H^1(D)^2\hookrightarrow L^q(D)^2$, for $q \in [1,6)$, further implies that $E_n(\widetilde{u}(\varphi_n),w(\varphi_n)) \to (u^*,w^*)$ in $L^q(D)^2$. 
		By \Cref{lemma:charcon} $(1-\chi_{{\Omega}_n})\to (1-\chi_{{\Omega}^*})$ in $L^2(D)$. From these facts, we get by virtue of Fatou's Lemma
		\begin{align*}
			\int_D (1-\chi_{{\Omega}^*}) |u^*|^2 \du x & \le \liminf_{n\to+\infty} \int_D (1-\chi_{{\Omega}_n}) |u^*|^2 \du x\\
			&  \le \liminf_{n\to+\infty} \int_D (1-\chi_{{\Omega}_n}) |u^* - \overline{u}_n|^2 \du x\\
			& \le \lim_{n\to+\infty} \|1 -\chi_{{\Omega}_n} \|_{L^2(D)}\| u^* - \overline{u}_n \|_{L^4(D)}\to 0.
		\end{align*}
		Here we used the notation $(\overline{u}_n,\overline{w}_n) = E_n(\widetilde{u}(\varphi_n),w(\varphi_n))$. Using the same argument for  $w^*$, we infer that $(u^*,w^*) = 0$ almost everywhere in $D\backslash{\Omega(\varphi^*)}$, whence $(u^*,w^*)\in H_0^1(\Omega^*)^2$.

		We now verify that $(u^*,w^*)\in H_0^1(\Omega^*)^2$ satisfies \eqref{weakform} with $\Omega = \Omega^*$. Let $(\psi_1,\psi_2)\in C_0^\infty(\Omega^*)^2$ be an arbitrary test pair. Since $\Omega^* = \displaystyle\bigcap_{m\in\mathbb{N}}\bigcup_{n\ge m}\Omega_n$, we see that there exists $m\in\mathbb{N}$ such that for each $n\ge m$ we get $(\psi_1,\psi_2)\in C_0^\infty(\Omega_n)^2$. For each $n\ge m$, we get
		\begin{align*}
			&\int_D \nabla \overline{u}_n\cdot \nabla \psi_1 \du x + \int_D \nabla \overline{w}_n\cdot \nabla \psi_2 \du x = \int_D \chi_{\Omega_n} f \psi_1 \du x - \int_D \chi_{\Omega_n}  \nabla u_{g_r}\cdot \nabla \psi_1 \du x + \int_\omega (\overline{u}_n + u_{g_r} -u_d)\psi_2 \du x .
		\end{align*}
		Passing to the limit with convergences we have in hand allows us to thus have
		\begin{align*}
			\int_D \nabla u^*\cdot \nabla \psi_1 \du x + \int_D \nabla w^*\cdot \nabla \psi_2 \du x = \int_D \chi_{\Omega^*} f \psi_1 \du x - \int_D \chi_{\Omega^*}  \nabla u_{g_r}\cdot \nabla \psi_1 \du x + \int_\omega (u^* + u_{g_r} -u_d)\psi_2 \du x.
		\end{align*}
		Because $(u^*,w^*) = 0$ almost everywhere in $D\backslash\overline{\Omega(\varphi^*)}$ and $(u^*,w^*)\in H_0^1(\Omega^*)^2$, we further infer that
		\begin{align}
			\int_{\Omega^*} \nabla u^*\cdot \nabla \psi_1 \du x + \int_{\Omega^*} \nabla w^*\cdot \nabla \psi_2 \du x = \int_{\Omega^*} f \psi_1 \du x - \int_{\Omega^*} \nabla u_{g_r}\cdot \nabla \psi_1 \du x + \int_\omega (u^* + u_{g_r} -u_d)\psi_2 \du x.
		\end{align}
		By the density of $C_0^\infty(\Omega^*)$ in $H_0^1(\Omega^*)$, we therefore conclude that $(u^*,w^*)_{|\Omega(\varphi^*)} = (u(\varphi^*),w(\varphi^*))$.
	\end{proof}
	
	\begin{remark}\label{ramark:continuityw}
		From \Cref{remark:regularity} $w(\varphi_n)\in H_0^1(\Omega(\varphi_n))\cap H^2(\Omega(\varphi_n))$, for any $n\in \mathbb{N}$. Furthermore, we observe that $\partial_\nu w(\varphi_n) = \partial_\nu \overline w_n$ on $\Sigma$, where $\overline w_n$ is such that $(\overline u_n,\overline w_n) = E_n (\widetilde{u}(\varphi_n),w(\varphi_n))$. According to trace theorem and \eqref{uniext}, we have 
		\begin{align*}
			\|\partial_\nu w(\varphi_n)\|_{H^{\frac{1}{2}}(\Sigma)} \le \|E(\widetilde{u}(\varphi_n),w(\varphi_n))\|_{H^2(D)^2} \le c ( \|f\|_{L^2(D)} +  \|g_r\|_{H^{\frac{3}{2}}(\Sigma)} + \|u_d\|_{L^2(\omega)} ).
		\end{align*}
		As a consequence, we get the strong convergence 
		\begin{align}\label{conv:wn}
			\partial_\nu w(\varphi_n)\to \partial_\nu w(\varphi^*) \text{ in } L^2(\Sigma)\text{ and a.e. on }\Sigma.
		\end{align} 
		Similarly, the convergence of the solutions $\widetilde{u}(\varphi_n)$ can be inherited by $u(\varphi_n) = \widetilde{u}(\varphi_n) + u_{g_r}$.
	\end{remark}

	\subsection{Existence of minimizing deformation fields}
	
	In this section, we will prove the existence of solutions to the no-regret and low-regret problems. We will first establish the existence of a solution to the low-regret problem because, as we will see later, the lower bound of the functional for the no-regret problem depends on the minimum value of the low-regret problem.
	
	Let us use the notation 
	\begin{align*}
		J_\varepsilon(\varphi) := \widetilde J(\varphi) - \widetilde J(\circphi) + \F_\varepsilon^*(\partial_\nu w(\circphi) - \partial_\nu w(\varphi))
	\end{align*}
	for the functional of the low-regret problem \eqref{low-regret}.
	The following theorem gives us the existence of the minimizing deformation field for the low-regret problem.
	\begin{theorem}
		Suppose that \Cref{assumption1} holds. There exists $\varphi_\varepsilon\in\Om(\Omega)$ that minimizes $J_\varepsilon$.
	\end{theorem}
	\begin{proof}
		For any $\varphi \in\Om(\Omega)$, we know that $- \widetilde J(\circphi) \le \widetilde J(\varphi)- \widetilde J(\circphi)$. Furthermore, we claim that $\F_\varepsilon^*(\partial_\nu w(\circphi) - \partial_\nu w(\varphi)) $ is uniformly bounded below by some constant. 
		
		Indeed, assuming without loss of regularity that the sets
		\begin{align}\label{segments}
			\begin{aligned}
				&\Sigma_1(\varphi) = \{s\in \Sigma : \partial_\nu w(\circphi)(s) - \partial_\nu w(\varphi)(s) < \varepsilon g_a \},\\
				&\Sigma_2(\varphi) = \{s\in \Sigma : \varepsilon g_a \le \partial_\nu w(\circphi)(s) - \partial_\nu w(\varphi)(s) \le \varepsilon g_b \},\\
				&\Sigma_3(\varphi) = \{s\in \Sigma : \varepsilon g_b < \partial_\nu w(\circphi)(s) - \partial_\nu w(\varphi)(s)  \},
			\end{aligned}
		\end{align}
		are of positive measure, we can write $\F_\varepsilon^*(\partial_\nu w(\circphi) - \partial_\nu w(\varphi))$ as follows
		\begin{align*}
			\F_\varepsilon^*(\partial_\nu w(\circphi) - \partial_\nu w(\varphi)) = \sum_{i=1}^3 \left\langle \partial_\nu w(\circphi) - \partial_\nu w(\varphi), \overline{g} \right\rangle_{\Sigma_i(\varphi)} - \frac{\varepsilon}{2}\|\overline{g}\|_{L^2(\Sigma_i(\varphi))}^2,
		\end{align*}
		where $\overline{g} = \arg\max_{g\in \G} \left\langle \partial_\nu w(\circphi) - \partial_\nu w(\varphi), {g} \right\rangle_{\Sigma} - \frac{\varepsilon}{2}\|{g}\|_{L^2(\Sigma)}^2$, which ---referring to  \Cref{dirichletKKT}--- is given by $$\overline{g} = P_{[g_a,g_b]}\left(\frac{1}{\varepsilon} (\partial_\nu w(\circphi) - \partial_\nu w(\varphi)) \right).$$
		Thus, by momentarily writing $W(\varphi) = \partial_\nu w(\circphi) - \partial_\nu w(\varphi)$, we have
		\begin{align*}
			\F_\varepsilon^*(W(\varphi)) =& \left\langle W(\varphi), {g}_a \right\rangle_{\Sigma_1(\varphi)} - \frac{\varepsilon}{2}\|{g}_a\|_{L^2(\Sigma_1(\varphi))}^2+ \frac{1}{2\varepsilon}\|W(\varphi)\|_{L^2(\Sigma_2(\varphi))}^2+ \left\langle  W(\varphi), {g}_b \right\rangle_{\Sigma_3(\varphi)} - \frac{\varepsilon}{2}\|{g}_b\|_{L^2(\Sigma_3(\varphi))}^2\\
			\ge & \left\langle (W(\varphi) - \varepsilon g_a), {g}_a \right\rangle_{\Sigma_1(\varphi)} + \frac{\varepsilon}{2}\|{g}_a\|_{L^2(\Sigma_1(\varphi))}^2 + \frac{1}{2\varepsilon}\|W(\varphi)\|_{L^2(\Sigma_2(\varphi))}^2 + \frac{\varepsilon}{2}\|{g}_b\|_{L^2(\Sigma_3(\varphi))}^2.
		\end{align*}
		As the terms on the right-hand side of the expression are non-negative, we infer that $J_\varepsilon(\varphi)$ is bounded from below for any $\varphi\in\Om(\Omega)$.
		We thus find a sequence $(\varphi_n)\subset  \Om(\Omega)$ such that
		\begin{align}
			\lim_{n\to+\infty} J_\varepsilon(\varphi_n) = \inf_{\varphi\in \Om(\Omega)}J_\varepsilon(\varphi) =: J_\varepsilon^*.
		\end{align}
		From the Arzel{\`a}-Ascoli theorem, we get the existence of $\varphi_\varepsilon\in\Om(\Omega)$ such that  --- up  to a subsequence --- the convergence $\varphi_n\to \varphi_\varepsilon$ in $C^2(\overline{D};\mathbb{R}^d)$ holds true. 
		
		From \Cref{shapecontinuity}, \Cref{ramark:continuityw} and the weak lower semicontinuity of the $L^2$ norm we get
		\begin{align}\label{lowsemconJ}
			\widetilde J(\varphi_\varepsilon) \le \liminf_{n\to+\infty} \widetilde J(\varphi_n).
		\end{align}
		For the term involving the Fenchel transform, we claim that the fact that $\partial_\nu w(\varphi_n)\to \partial_\nu w(\varphi_\varepsilon)$ in $L^2(\Sigma)$ and a.e. on $\Sigma$ implies 
		\begin{align}\label{confenchel}
			\F_\varepsilon^*(W(\varphi_n)) \to \F_\varepsilon^*(W(\varphi_\varepsilon)).
		\end{align}
		
		Indeed, we note that due to the constraints on the elements of $\G$ we find $c_{\G}>0$ such that $\|g\|_{L^2(\Sigma)} \le c_{\G}$ for any $g\in\G$. Let $\overline{g}(\varphi_n) = P_{[g_a,g_b]}(\frac{1}{\varepsilon}W(\varphi_n))$ and $\overline{g}(\varphi_\varepsilon) = P_{[g_a,g_b]}(\frac{1}{\varepsilon}W(\varphi_\varepsilon))$ be the maximizers of $\langle W(\varphi_n),g \rangle_\Sigma - \frac{\varepsilon}{2}\|g\|^2_{L^2(\Sigma)}$ and $\langle W(\varphi_\varepsilon),g \rangle_\Sigma - \frac{\varepsilon}{2}\|g\|^2_{L^2(\Sigma)}$, respectively. This gives us
		\begin{align*}
			&\F_\varepsilon^*(W(\varphi_n)) - \F_\varepsilon^*(W(\varphi_\varepsilon)) = \langle W(\varphi_n),\overline{g}(\varphi_n) \rangle_\Sigma - \frac{\varepsilon}{2}\|\overline{g}(\varphi_n)\|^2_{L^2(\Sigma)} - \left(\langle W(\varphi_\varepsilon),\overline{g}(\varphi_\varepsilon) \rangle_\Sigma - \frac{\varepsilon}{2}\|\overline{g}(\varphi_\varepsilon)\|^2_{L^2(\Sigma)}\right)\\
			& \le \langle \partial_\nu w(\varphi_\varepsilon) - \partial_\nu w(\varphi_n),\overline{g}(\varphi_n) \rangle_\Sigma \le c_{\G}\|\partial_\nu w(\varphi_\varepsilon) - \partial_\nu w(\varphi_n)\|_{L^2(\Sigma)}.
		\end{align*}
		Similarly, we have
		\begin{align*}
			&\F_\varepsilon^*(W(\varphi_\varepsilon)) - \F_\varepsilon^*(W(\varphi_n)) = \langle W(\varphi_\varepsilon),\overline{g}(\varphi_\varepsilon) \rangle_\Sigma - \frac{\varepsilon}{2}\|\overline{g}(\varphi_\varepsilon)\|^2_{L^2(\Sigma)} -  \left(\langle W(\varphi_n),\overline{g}(\varphi_n) \rangle_\Sigma - \frac{\varepsilon}{2}\|\overline{g}(\varphi_n)\|^2_{L^2(\Sigma)}\right)\\
			& \le \langle \partial_\nu w(\varphi_\varepsilon) - \partial_\nu w(\varphi_n),\overline{g}(\varphi_n) \rangle_\Sigma \le c_{\G}\|\partial_\nu w(\varphi_\varepsilon) - \partial_\nu w(\varphi_n)\|_{L^2(\Sigma)}.
		\end{align*}
		{The estimates above and \eqref{conv:wn} imply \eqref{confenchel}.}

		Combining \eqref{lowsemconJ} and \eqref{confenchel} implies that $\varphi_\varepsilon\in\Om(\Omega)$ is a minimizer of $J_\varepsilon$.
	\end{proof}

	For the no-regret problem, we note that due to the linearity of the operator $[g\mapsto\langle \partial_\nu w(\circphi)-\partial_\nu w(\varphi), g \rangle_\Sigma] : \G \to \mathbb{R}$ and the constraints in $\G$, given $\varphi\in\Om(\Omega)$ we find $g^*(\varphi)\in \G$ such that $$g^*(\varphi) = \displaystyle\arg\max_{g_\delta\in\G} \langle\partial_\nu w(\circphi)-\partial_\nu w(\varphi), g_\delta \rangle_\Sigma.$$
	
	So the no-regret problem \eqref{no-regret} is now translated into finding the minimizing deformation field for the functional 
	\begin{align*}
		J^*(\varphi) := \widetilde  J(\varphi)- \widetilde J(\circphi) + \langle\partial_\nu w(\circphi)-\partial_\nu w(\varphi), g^*(\varphi) \rangle_\Sigma.
	\end{align*}
	The theorem below asserts the existence of a minimizing element.
	\begin{theorem}
		Suppose that \Cref{assumption1} holds. Then there exists $\varphi^*\in\Om(\Omega)$ that minimizes $J^*$ and hence the no-regret problem admits a solution.
	\end{theorem}
	
	\begin{proof}
		We get the lower bound for $J^*(\varphi)$ as follows: first by definition of $g^*(\varphi)\in\G$ 
		\begin{align*}
			J^*(\varphi) \ge \widetilde J(\varphi)- \widetilde J(\circphi) + \langle \partial_\nu w(\circphi)-\partial_\nu w(\varphi), g \rangle_\Sigma - \frac{\varepsilon}{2}\|g\|_{L^2(\Sigma)}^2 \quad\forall g\in\G.
		\end{align*}
		By choosing $g = P_{[g_a,g_b]}\left(\frac{1}{\varepsilon}(\partial_\nu w(\circphi)-\partial_\nu w(\varphi)) \right)\in \G$ {and using the fact that $\varphi_\varepsilon$ is a minimizer for $J_\varepsilon$} we get 
		\begin{align}\label{noglo}
			\begin{aligned}
				J^*(\varphi) &\ge \widetilde J(\varphi)- \widetilde J(\circphi) + \F_\varepsilon^*(\partial_\nu w(\circphi)-\partial_\nu w(\varphi)) \\
				&\ge  \widetilde J(\varphi_\varepsilon) - \widetilde  J(\circphi) + \F_\varepsilon^*(\partial_\nu w(\circphi)-\partial_\nu w(\varphi_\varepsilon)) = J_\varepsilon(\varphi_\varepsilon).
			\end{aligned}
		\end{align}
		
		Having been able to establish the boundedness of $J^*$ from below, we find an infimizing sequence $(\varphi_n)\subset  \Om(\Omega)$ such that
		\begin{align}
			\lim_{n\to+\infty} J^*(\varphi_n) = \inf_{\varphi\in \Om(\Omega)}J^*(\varphi)  .
		\end{align}
		
		Again, the Arzel{\'a}-Ascoli theorem implies that, up to a subsequence, $\varphi_n \to \varphi^*$ in $C^2(\overline{D};\mathbb{R}^d)$ for some $\varphi^*\in \Om(\Omega)$. The weak lower semicontinuity of the $L^2$ norm and  \Cref{shapecontinuity} imply
		\begin{align*}
			\widetilde J(\varphi^*) - \widetilde J(\circphi) \le \liminf_{n\to+\infty} \widetilde J(\varphi_n) - \widetilde J(\circphi).
		\end{align*}
		\Cref{ramark:continuityw} implies $\langle\partial_\nu w(\circphi)-\partial_\nu w(\varphi_n), g^*(\varphi_n) \rangle_\Sigma \to \langle\partial_\nu w(\circphi)-\partial_\nu w(\varphi^*), g^*(\varphi^*) \rangle_\Sigma$.  Indeed, since $g^*(\varphi_n) = \displaystyle\arg\max_{g\in\G} \langle\partial_\nu w(\circphi)-\partial_\nu w(\varphi_n), g \rangle_\Sigma,$ we get
		\begin{align*}
			& \langle\partial_\nu w(\circphi)-\partial_\nu w(\varphi^*), g^*(\varphi^*) \rangle_\Sigma - \langle\partial_\nu w(\circphi)-\partial_\nu w(\varphi_n), g^*(\varphi_n) \rangle_\Sigma 
			\le \langle \partial_\nu w(\varphi_n)-\partial_\nu w(\varphi^*), g^*(\varphi^*) \rangle_\Sigma\\
			& \le c_{\G}\|\partial_\nu w(\varphi_n)-\partial_\nu w(\varphi^*)\|_{L^2(\Sigma)}.
		\end{align*}
		Similarly, because $g^*(\varphi^*) = \displaystyle\arg\max_{g\in\G} \langle\partial_\nu w(\circphi)-\partial_\nu w(\varphi^*), g \rangle_\Sigma,$ we have
		\begin{align*}
			& \langle\partial_\nu w(\circphi) - \partial_\nu w(\varphi_n), g^*(\varphi_n) \rangle_\Sigma - \langle\partial_\nu w(\circphi) - \partial_\nu w(\varphi^*), g^*(\varphi^*) \rangle_\Sigma 
			\le \langle \partial_\nu w(\varphi^*)-\partial_\nu w(\varphi_n), g^*(\varphi_n) \rangle_\Sigma\\
			& \le c_{\G}\|\partial_\nu w(\varphi_n)-\partial_\nu w(\varphi^*)\|_{L^2(\Sigma)}.
		\end{align*}
		The computations above sum up to the following estimate
		\begin{align*}
			|\langle\partial_\nu w(\circphi) - \partial_\nu w(\varphi_n), g^*(\varphi_n) \rangle_\Sigma - \langle\partial_\nu w(\circphi) - \partial_\nu w(\varphi^*), g^*(\varphi^*) \rangle_\Sigma| \le c_{\G}\|\partial_\nu w(\varphi_n)-\partial_\nu w(\varphi^*)\|_{L^2(\Sigma)}.
		\end{align*}
		The right-hand side goes to zero per \Cref{ramark:continuityw}, which proves our claim.
		Therefore, $\varphi^*$ is a minimizer for $J^*$.
	\end{proof}

	\section{From low-regret to no-regret}\label{section:low-to-no}
	
	In this section, we provide one of the main results of this paper. That is, we show that as $\varepsilon\to 0$, the sequence $(\varphi_\varepsilon)\subset \Om(\Omega)$ of minimizing deformations for the low-regret problem converges to a deformation field that minimizes $J^*$.
	
	We begin by looking at the gap between the minimum values of $J_\varepsilon$ and $J^*$.
	\begin{lemma}\label{lemma:low-nogap}
		Suppose that $\varphi_\varepsilon\in\Om(\Omega)$ and $\varphi^*\in\Om(\Omega)$ are minimizers for $J_\varepsilon$ and $J^*$, respectively. Then we have the following estimate
		\begin{align}
			|J^*(\varphi^*) - J_\varepsilon(\varphi_\varepsilon)| \le c {\varepsilon},
		\end{align}
		for some $c>0$ not dependent on $\varepsilon>0$.
	\end{lemma}
	\begin{proof}
		Let us first underline the fact that, due to \eqref{noglo}, $J^*(\varphi^*) \ge J_\varepsilon(\varphi_\varepsilon)$. This implies that
		\begin{align*}
			| J^*(\varphi^*) & - J_\varepsilon(\varphi_\varepsilon) | = J^*(\varphi^*) - J_\varepsilon(\varphi_\varepsilon) \le J^*(\varphi_\varepsilon) - J_\varepsilon(\varphi_\varepsilon)\\ =&\, \| u(\varphi_\varepsilon) - u_d \|_{L^2(\omega)}^2 + \langle\partial_\nu w(\circphi)-\partial_\nu w(\varphi_\varepsilon), g^*(\varphi_\varepsilon) \rangle_\Sigma \\
			&- \left( \| u(\varphi_\varepsilon) - u_d \|_{L^2(\omega)}^2 + \langle\partial_\nu w(\circphi)-\partial_\nu w(\varphi_\varepsilon), \overline{g}(\varphi_\varepsilon) \rangle_\Sigma - \frac{\varepsilon}{2}\|\overline{g}(\varphi_\varepsilon)\|_{L^2(\Sigma)}^2 \right)\\
			=&\, \langle\partial_\nu w(\circphi) - \partial_\nu w(\varphi_\varepsilon), g^*(\varphi_\varepsilon) \rangle_\Sigma - \left(\langle\partial_\nu w(\circphi)-\partial_\nu w(\varphi_\varepsilon), \overline{g}(\varphi_\varepsilon) \rangle_\Sigma - \frac{\varepsilon}{2}\|\overline{g}(\varphi_\varepsilon)\|_{L^2(\Sigma)}^2 \right),
		\end{align*}
		where
		\begin{align*}
			& g^*(\varphi_\varepsilon) = \arg\max_{g_\delta\in\G}\langle\partial_\nu w(\circphi)-\partial_\nu w(\varphi_\varepsilon), g_\delta\rangle_\Sigma \; \text{ and }\\
			& \overline{g}(\varphi_\varepsilon) = \arg\max_{g_\delta\in\G}\left(  \langle\partial_\nu w(\circphi)-\partial_\nu w(\varphi_\varepsilon), g_\delta \rangle_\Sigma - \frac{\varepsilon}{2}\|g_\delta\|_{L^2(\Sigma)}^2 \right).
		\end{align*}
		The desired estimate is obtained as follows:
		\begin{align*}
			| J^*(\varphi^*)  - J_\varepsilon(\varphi_\varepsilon) | \le \langle\partial_\nu w(\circphi) - \partial_\nu w(\varphi_\varepsilon), g^*(\varphi_\varepsilon) \rangle_\Sigma - \left( \langle\partial_\nu w(\circphi) - \partial_\nu w(\varphi_\varepsilon), \overline{g}(\varphi_\varepsilon) \rangle_\Sigma - \frac{\varepsilon}{2}\|\overline{g}(\varphi_\varepsilon)\|_{L^2(\Sigma)}^2 \right)\\
			\le  \langle\partial_\nu w(\circphi) - \partial_\nu w(\varphi_\varepsilon), g^*(\varphi_\varepsilon) \rangle_\Sigma - \left( \langle\partial_\nu w(\circphi)-\partial_\nu w(\varphi_\varepsilon), g^*(\varphi_\varepsilon)  \rangle_\Sigma - \frac{\varepsilon}{2}\|g^*(\varphi_\varepsilon) \|_{L^2(\Sigma)}^2 \right) \le \frac{c_{\G}^2}{2}\varepsilon.
		\end{align*}
	\end{proof}
	
	Due to the compactness of $\Om(\Omega)$ with respect to the $C^2(\overline{D};\mathbb{R}^d)$-topology, we see that the sequence $(\varphi_\varepsilon)_{\varepsilon>0}\subset \Om(\Omega)$ has a subsequence converging to some $\varphi_0\in \Om(\Omega)$, i.e. $\varphi_\varepsilon\to \varphi_0$ in $\Om(\Omega)$ as $\varepsilon\to 0$. Our main result shows that $\varphi_0\in\Om(\Omega)$ minimizes $J^*$.
	\begin{theorem}
		The deformation field $\varphi_0\in\Om(\Omega)$ is a minimizer for the no-regret problem.
	\end{theorem}
	\begin{proof}
		Knowing that $\varphi^*$ is a minimizer for $J^*$, we prove that $\varphi_0$ is also a minimizer by showing that $J^*(\varphi_0) = J^*(\varphi^*)$. Let
		\begin{align*}
			g^*(\varphi_0) = \arg\max_{g_\delta\in\G}\langle\partial_\nu w(\circphi)-\partial_\nu w(\varphi_0), g_\delta \rangle_\Sigma.
		\end{align*}
		Hence, we have by \Cref{lemma:low-nogap}
		\begin{align*}
			|J^*(\varphi^*) - J^*(\varphi_0)| \le |J^*(\varphi^*) - J_\varepsilon(\varphi_\varepsilon)| + |J_\varepsilon(\varphi_\varepsilon) - J^*(\varphi_0)| \le c\varepsilon + |J_\varepsilon(\varphi_\varepsilon) - J^*(\varphi_0)|.
		\end{align*}
		The following computations give us estimates for the second term on the right-hand side above: first, we obtain
		\begin{align*}
			J_\varepsilon(\varphi_\varepsilon) - J^*(\varphi_0) & =  \| u(\varphi_\varepsilon) - u_d \|_{L^2(\omega)}^2 + \langle\partial_\nu w(\circphi)-\partial_\nu w(\varphi_\varepsilon), \overline{g}(\varphi_\varepsilon) \rangle_\Sigma - \frac{\varepsilon}{2}\|\overline{g}(\varphi_\varepsilon)\|_{L^2(\Sigma)}^2\\
			&\quad - \left(\| u(\varphi_0) - u_d \|_{L^2(\omega)}^2 + \langle\partial_\nu w(\circphi)-\partial_\nu w(\varphi_0), g^*(\varphi_0) \rangle_\Sigma \right)\\
			& \le \langle u(\varphi_\varepsilon) - u(\varphi_0), u(\varphi_\varepsilon) + u(\varphi_0) - 2u_d \rangle_{\omega} +  \langle\partial_\nu w(\varphi_0)-\partial_\nu w(\varphi_\varepsilon), \overline{g}(\varphi_\varepsilon) \rangle_\Sigma\\
			& \le c\| u(\varphi_\varepsilon) - u(\varphi_0)\|_{L^2(\omega)} +  c_{\G} \|\partial_\nu w(\varphi_0)-\partial_\nu w(\varphi_\varepsilon)\|_{L^2(\Sigma)},
		\end{align*}
		where $c>0$ is a constant dependent only on the terms in the right hand side of \eqref{estimate:sol} and is independent of $\varepsilon$. Similarly, we have
		\begin{align*}
			J^*(\varphi_0) - J_\varepsilon(\varphi_\varepsilon) & =  \| u(\varphi_0) - u_d \|_{L^2(\omega)}^2 + \langle\partial_\nu w(\circphi)-\partial_\nu w(\varphi_0), g^*(\varphi_0) \rangle_\Sigma\\
			&\quad - \left(\| u(\varphi_\varepsilon) - u_d \|_{L^2(\omega)}^2 + \langle\partial_\nu w(\circphi)-\partial_\nu w(\varphi_\varepsilon), \overline{g}(\varphi_\varepsilon) \rangle_\Sigma - \frac{\varepsilon}{2}\|\overline{g}(\varphi_\varepsilon)\|_{L^2(\Sigma)}^2 \right)\\
			& \le \langle u(\varphi_0) - u(\varphi_\varepsilon), u(\varphi_0) + u(\varphi_\varepsilon) - 2u_d \rangle_{\omega} +  \langle\partial_\nu w(\varphi_\varepsilon)-\partial_\nu w(\varphi_0), g^*(\varphi_0) \rangle_\Sigma + \frac{c_{\G}^2}{2} \varepsilon\\
			& \le c\| u(\varphi_0) - u(\varphi_\varepsilon)\|_{L^2(\omega)} +  c_{\G} \|\partial_\nu w(\varphi_\varepsilon)-\partial_\nu w(\varphi_0)\|_{L^2(\Sigma)} + \frac{c_{\G}^2}{2} \varepsilon,
		\end{align*}
		where $c>0$, as in the previous case, is not dependent on $\varepsilon$.
		Summing up the estimates imply
		\begin{align*}
			|J^*(\varphi^*) - J^*(\varphi_0)| \le c\| u(\varphi_0) - u(\varphi_\varepsilon)\|_{L^2(\omega)} +  c_{\G} \|\partial_\nu w(\varphi_\varepsilon)-\partial_\nu w(\varphi_0)\|_{L^2(\Sigma)} + \frac{c_{\G}^2}{2} \varepsilon.
		\end{align*}
		Taking $\varepsilon\to 0$, {from \Cref{shapecontinuity} and \eqref{conv:wn} the expression on the right-hand side tends to zero.}
	\end{proof}
	
	\section{Optimality condition for the low-regret problem}\label{section:derivative}

	For a fixed $\Omega\subset D$ we consider the functional $j:\Om(\Omega) \subset C^{1,1}(\overline{D};\mathbb{R}^d)\to \mathbb{R}$. {We say that it has a G{\^ a}teaux derivative at $\varphi\in C^{1,1}(\overline{D};\mathbb{R}^d)$ in the direction $\delta\varphi\in C^{1,1}(\overline{D};\mathbb{R}^d)$ if the following limit exists:
	\begin{align}
		d_\varphi j(\varphi)\delta\varphi = \lim_{t \searrow 0} \frac{ j(\varphi + t\delta\varphi) - j(\varphi)}{t}.
	\end{align}
	The functional $j$ is said to be G{\^a}teaux differentiable at $\varphi\in C^{1,1}(\overline{D};\mathbb{R}^d)$ if the derivative exists for any direction $\delta\varphi\in C^{1,1}(\overline{D};\mathbb{R}^d)$ and the map $\delta\varphi\mapsto d_\varphi j(\varphi)\delta\varphi $ is linear and bounded.}
	
	{We note that the derivative defined above has been discussed in \cite[Chapter 9 Section 3.3]{delfour2011}, and can be related to the so-called shape derivatives---both in the context Eulerian and Hadamard semiderivatives---for shape functionals, see \cite[Theorem 9.3.4]{delfour2011}.}

	Let $\Omega \in \mathcal{P}_{\Sigma} := \{ \Omega\in\mathcal{P}(D): \Omega\text{ is an annular domain with }\Sigma\subset\partial\Omega \text{ and of class }C^{2,1}  \}$. With \eqref{system:couple} in mind, we define the operator $E[\Omega]: H_0^1(\Omega)^2\to H^{-1}(\Omega)^2 $ by
	\begin{align*}
		\langle E [\Omega] (\widetilde{u},w) , (\psi_1,\psi_2)\rangle_{H_0^1(\Omega)^2}
		= \int_\Omega \nabla \widetilde{u}\cdot \nabla \psi_1 \du x + \int_\Omega \nabla w\cdot \nabla \psi_2 \du x &\\  - \int_\Omega f \psi_1 \du x + \int_\Omega \nabla u_{g_r}\cdot\nabla \psi_1\du x- \int_\omega (\widetilde{u} + u_{g_r} - u_d)\psi_2 \du x.&
	\end{align*}
	We can thus define the design-to-state operator $\mathcal{S}: \mathcal{P}_{\Sigma} \to H^1(D)^2$ by $\mathcal{S}(\Omega) = (\widetilde{u},w)\in H_0^1(\Omega)^2\subset H^1(D)^2$ if and only if $E[\Omega](\widetilde{u},w) = 0$ in $H^{-1}(\Omega)^2$.
	\medskip
	
	\noindent\textbf{Function space parametrization.} Let $\Omega \subset D$ be fixed. For any $\varphi\in C^{2,1}(\overline{D};\mathbb{R}^d)$ with $\|\varphi\|_{C^{2,1}} \le c <1$ we use the transformation $\tau_\varphi = I + \varphi$ and the notation $\Omega(\varphi) = \tau_\varphi(\Omega)$, and define the function space parametrizations (see c.f. \cite[Section 2.3.1 Lemmata 3.1--3.2]{necas2012} and \cite[Section 2.3.6 Corollary 3.1]{necas2012})
	\begin{align*}
		H_0^1(\Omega) = \{ \overline{\psi}\circ\tau_\varphi : \overline{\psi}\in H_0^1(\Omega(\varphi))  \},\,
		L^2(\Omega) = \{ \overline{\psi}\circ\tau_\varphi : \overline{\psi}\in L^2(\Omega(\varphi))  \}.
	\end{align*}
	We now define the transformed operator $E_\Omega: \Om(\Omega)\times H_0^1(\Omega)^2 \to H^{-1}(\Omega)^2$ as $$E_{\Omega}[\varphi,(\tilde{u},w)] = (E[\tau_\varphi(\Omega)] (\tilde{u}\circ\tau_\varphi^{-1},w\circ\tau_\varphi^{-1}))\circ \tau_\varphi.$$
	We can  express  the action of $E_\Omega$ as follows:
	\begin{align}\label{operator:Eom}
		\begin{aligned}
		\langle E_{\Omega}[\varphi,(\tilde{u},w)],(\psi_1,\psi_2) \rangle_{H_0^1(\Omega)^2} = \int_{\Omega} A(\varphi)\nabla \tilde{u}\cdot \nabla\psi_1 \du x + \int_{\Omega} A(\varphi)\nabla w\cdot \nabla\psi_2 \du x &\\
		- \int_\Omega b(\varphi)\hf(\varphi) \psi_1 \du x + \int_{\Omega}A(\varphi)\nabla \hu_{g_r}(\varphi)\cdot\nabla \psi_1 - \int_\omega b(\varphi)( \tilde{u} + \hu_{g_r}(\varphi) - u_d)\psi_2 \du x ,&
		\end{aligned}
	\end{align}
	for $(\psi_1,\psi_2)\in H_0^1(\Omega)^2$, where $A(\varphi) = b(\varphi)D\tau_\varphi^{-1}(D\tau_\varphi^{-1})^{\top}$, $b(\varphi) = \mathrm{det}(D\tau_\varphi)$, $\hf(\varphi) = f\circ\tau_\varphi$ and $\hu_{g_r}(\varphi) = u_{g_r}\circ\tau_\varphi$.

	To proceed we impose additional assumptions on the source function $f$, the given Dirichlet data $g_r$, and the desired profile $u_d$.
	\begin{assumption}\label{assumption2}
		We assume that $f\in H^1(D)$ and $g_r \in H^{\frac{5}{2}}(\Sigma)$, which implies that $u_{g_r}$ is in $H^{3}(D)$. Furthermore, we assume that $(u_d\circ \tau_\varphi)_{|\omega} = u_d$, for all $\varphi\in\Om(\Omega)$.
	\end{assumption}
	
	We now define the transformed deformation-to-state map $\mathcal{S}_\Omega : \Om(\Omega)\to H_0^1(\Omega)^2$ as $\mathcal{S}_\Omega(\varphi) = \mathcal{S}(\Omega(\varphi))\circ \tau_\varphi \in H_0^1(\Omega)^2$,  which  satisfies $E_\Omega[\varphi,\mathcal{S}_\Omega(\varphi)] = 0$ in $H^{-1}(\Omega)^2$. Indeed, we have
	\begin{align*}
		E_\Omega[\varphi,\mathcal{S}_\Omega(\varphi)] = E[\Omega(\varphi)] \mathcal{S}(\Omega(\varphi)) \circ \tau_\varphi = 0.
	\end{align*}

{Taking into consideration \eqref{operator:Eom} and the differentiability of $A(\varphi)$, $b(\varphi)$, $\widehat{f}(\varphi)$ and $\widehat{u}_g(\varphi)$ it can be shown that the operator $E_\Omega$ is continuously Fr{\'e}chet differentiable in both of its arguments.} 
	In fact, its derivative with respect to the first argument at $\varphi\in \Om(\Omega)$ in  direction $\delta\varphi \in \Om(\Omega)$ can be written as
	\begin{align*}
		\langle d_\varphi E_\Omega[\varphi,(\tilde{u},w)] \delta \varphi, (\psi_1,\psi_2) \rangle_{H_0^1(\Omega)^2} = \int_{\Omega} A'(\varphi)\delta\varphi \nabla\tilde{u}\cdot \nabla\psi_1 \du x + \int_{\Omega} A'(\varphi)\delta\varphi \nabla w\cdot \nabla\psi_2 \du x&\\
		- \int_\Omega b'(\varphi)\hf(\varphi) \psi_1 \du x + \int_{\Omega}A'(\varphi)\nabla \hu_{g_r}(\varphi)\cdot\nabla \psi_1 - \int_\omega b'(\varphi)( \tilde{u} + \hu_{g_r}(\varphi) - u_d)\psi_2 \du x&\\
		- \int_\Omega b(\varphi)\hf'(\varphi)\delta\varphi \psi_1 \du x + \int_{\Omega}A(\varphi)\nabla [\hu_{g_r}'(\varphi)\delta\varphi]\cdot\nabla \psi_1 - \int_\omega b(\varphi)\hu_{g_r}'(\varphi)\delta\varphi \psi_2 \du x ,&
	\end{align*}
	where $\hf'(\varphi)\delta\varphi = (D\tau_\varphi^{-1})^{\top}\nabla\hf(\varphi)\cdot\delta\varphi \in L^2(D)$ and $\hu_{g_r}'(\varphi)\delta\varphi = (D\tau_\varphi^{-1})^{\top}\nabla\hu_{g_r}(\varphi)\cdot\delta\varphi \in H_0^1(\Omega)$. Its derivative with respect to its second argument at $(\tilde{u},w)\in H_0^1(\Omega)^2$ in  direction $(\delta u,\delta w)\in H_0^1(\Omega)^2$ is given by
	\begin{align}\label{Eder2}
		\langle &d_{(\tilde{u},w)} E_\Omega[\varphi,(\tilde{u},w)] (\delta u, \delta w), (\psi_1,\psi_2) \rangle_{H_0^1(\Omega)^2}\\
		& = \int_{\Omega} A(\varphi)\nabla\delta u\cdot \nabla\psi_1 \du x + \int_{\Omega} A(\varphi)\nabla\delta w\cdot \nabla\psi_2 \du x - \int_\omega b(\varphi) \delta u \psi_2 \du x.
	\end{align}
	Above, we have $A'(\varphi)\delta\varphi = D\tau_\varphi^{-1} [ -D\delta\varphi D\tau_\varphi^{-1} - (D\tau_\varphi^{-1})^\top (D\delta\varphi)^{\top} + I \mathrm{tr}(D\tau_\varphi^{-1} D\delta\varphi )] (D\tau_\varphi^{-1})^{\top}\mathrm{det}(D\tau_\varphi)$ and $b'(\varphi)\delta\varphi = \mathrm{tr}(D\tau_\varphi^{-1} D\delta\varphi )\mathrm{det}(D\tau_\varphi)$, where we use Jacobi's formula.

	From the discussion above, we now obtain the differentiability of the deformation-to-state operator.
	\begin{proposition}
		Suppose that Assumptions \ref{assumption1} and \ref{assumption2} hold. Then the map $\mathcal{S}_\Omega : \Om(\Omega)^2\to H_0^1(\Omega)^2$ is differentiable. Its first derivative at $\varphi\in\Om(\Omega)$ in the direction $\delta\varphi\in\Om(\Omega)$ is given by $(z,v) = S_\Omega'(\varphi)\delta\varphi$ which solves the system
		\begin{align}\label{system:linear}
			\begin{aligned}
				-\dive(A(\varphi)\nabla(z + \hu_{g_r}'(\varphi)\delta\varphi)) & = b'(\varphi)\delta\varphi \hf + b(\varphi)\hf'(\varphi)\delta\varphi + \dive(A'(\varphi)\delta\varphi\nabla u) && \text{in }\Omega,\\
				-\dive(A(\varphi)\nabla v) - \chi_\omega b(\varphi)(z + \hu_{g_r}'(\varphi)\delta\varphi) & = b'(\varphi)\delta\varphi(u-u_d)\chi_\omega + \dive(A'(\varphi)\delta\varphi\nabla w)&&\text{in }\Omega,\\
				z = 0,\, v& = 0 &&\text{on }\partial\Omega,
			\end{aligned}
		\end{align}
		where $u = \tilde{u}+\hu_{g_r} \in H^2(\Omega)$, and $(\tilde{u},w) = S_\Omega(\varphi)$.
	\end{proposition}
	\begin{proof}
		Let $\varphi^0\in \Om(\Omega)$ and $(u^0,w^0)\in H_0^1(\Omega)$ be such that $E_\Omega[\varphi^0,(u^0,w^0)] = 0$.
		Since the operators $A$ and $b$ are of class ${C}^{\infty}$, we see that $E_\Omega \in C^{\infty}(\Om(\Omega)\times H_0^1(\Omega)^2; H^{-1}(\Omega)^2)$.
		Following similar arguments as in \Cref{prop:stateexistence}, $d_{(u,w)}E_\Omega[\varphi^0,(u^0,w^0)] \in \mathcal{L}(H_0^1(\Omega)^2; H^{-1}(\Omega)^2)$ is an isomorphism.
		The implicit function theorem (see e.g. \cite{pata2019}) thus imply that there exist neighborhoods $\Om_{0} \subset \Om(\Omega)$ and $V_{0}\subset H_0^1(\Omega)^2$ centered at $\varphi^0$ and $(u^0,w^0)$, and an operator $\widetilde{\mathcal{S}}_0\in C^{\infty}(\Om_{0};V_0)$ such that $E_\Omega[\varphi,\widetilde{\mathcal{S}}_0(\varphi)] = 0$ for all $\varphi\in \Om_0$.
		From the definition of the deformation-to-state operator $\mathcal{S}_\Omega = \widetilde{\mathcal{S}}_0$ in $\Om_0$. From the arbitrary nature of $\varphi^0\in \Om(\Omega)$ and $(u^0,w^0)\in H_0^1(\Omega)^2$, $\mathcal{S}_\Omega$ is differentiable.
		
		Using chain rule, we thus get $\mathcal{S}_{\Omega}'(\varphi)\delta\varphi = -(d_{(\tilde{u},w)}E_\Omega[\varphi,\mathcal{S}_\Omega(\varphi)])^{-1} d_\varphi E_\Omega[\varphi,\mathcal{S}_\Omega(\varphi)]\delta\varphi$ which can be easily shown, by integration by parts, solves \eqref{system:linear}.
	\end{proof}

	To aid us in solving the derivative of the objective functional, we consider the following adjoint system:
	\begin{align}\label{adjoint:eulerian}
		\begin{aligned}
			&\langle d_{(\tilde{u},w)}E_\Omega[\varphi,\mathcal{S}_\Omega(\varphi)](\psi_1,\psi_2), (p,q) \rangle_{H_0^1(\Omega)^2} \\ & = (\chi_\omega b(\varphi)(\mathcal{S}_{\Omega,1}(\varphi) + \hu_{g_r}(\varphi) - u_d),\psi_1 )_\Omega - (A(\varphi)\nabla\psi_2,\nabla \overline{q}) + (b(\varphi)\overline{q}\chi_\omega,\psi_1),
		\end{aligned}
	\end{align}
	for all $(\psi_1,\psi_2)\in H_0^1(\Omega)$, where $(\mathcal{S}_{\Omega,1}(\varphi),\mathcal{S}_{\Omega,2}(\varphi)) = \mathcal{S}_{\Omega}(\varphi)$, $\overline{q}\in H^{1}(\Omega)$ is such that $\overline{q}_{|\Gamma} = 0$ and
	\begin{align}\label{qbar}
		\overline{q}_{|\Sigma} = - P_{[g_a,g_b]}\left(\frac{1}{\varepsilon}[\partial_\nu \mathcal{S}_{\Omega,2}(\circphi) - \partial_{\nu}\mathcal{S}_{\Omega,2}(\varphi)]\right)\in H^{\frac{1}{2}}(\Sigma).
	\end{align}
	Note that -- with similar arguments as in \Cref{ramark:continuityw} -- $\mathcal{S}_{\Omega,2}(\varphi), \partial_\nu \mathcal{S}_{\Omega,2}(\circphi)\in H_0^1(\Omega)\cap H^3(\Omega)$ and thus $ \partial_\nu \mathcal{S}_{\Omega,2}(\circphi) - \partial_\nu \mathcal{S}_{\Omega,2}(\varphi)\in H^{\frac{3}{2}}(\Sigma)$.

	Now, we write the objective functional $J_\varepsilon$ as follows:
	\begin{align*}
		J_\varepsilon(\varphi) = &\, \frac{1}{2}\int_\Omega \chi_\omega |\mathcal{S}_{\Omega,1}(\varphi) + \hu_{g_r}(\varphi)  - u_d |^2 b(\varphi) \du x - \frac{1}{2}\int_\Omega \chi_\omega |\mathcal{S}_{\Omega,1}(\circphi) + \hu_{g_r}(\circphi) - u_d |^2 b(\circphi) \du x\\ & + \F_\varepsilon^*( \partial_\nu \mathcal{S}_{\Omega,2}(\circphi) - \partial_\nu \mathcal{S}_{\Omega,2}(\varphi) ) .
	\end{align*}
	
	Before we proceed in the computation of the derivative of the functional $J_\varepsilon$, let us discuss the Fr{\'e}chet differentiability of the Fenchel transform $\F_\varepsilon^*$.
	\begin{lemma}
		For $y\in L^2(\Sigma)$, the Fenchel transformation  $\F_\varepsilon^*$ has a Fr{\'e}chet derivative at $y$ denoted as $d\F_\varepsilon^*(y) : L^2(\Sigma)\to \mathbb{R}$ whose action is given by
		\begin{align*}
			d\F_\varepsilon^*(y) \delta y=\left\langle P_{[q_a,q_b]}\left( \frac{1}{\varepsilon}y \right), \delta y \right\rangle_\Sigma \quad \forall \delta y\in L^2(\Sigma).
		\end{align*}
	\end{lemma}
	\begin{proof}
		Let us define $\overline{g} = P_{[g_a,g_b]}(\frac{1}{\varepsilon}y)$ and $\overline{g}_\delta = P_{[g_a,g_b]}(\frac{1}{\varepsilon}(y+\delta y))$ so that $\F_\varepsilon^*(y) = \langle y,\overline{g}\rangle_\Sigma - \frac{\varepsilon}{2}\|\overline{g}\|_{L^2(\Sigma)}^2$ and $\F_\varepsilon^*(y+\delta y) = \langle y+\delta y,\overline{g}_\delta\rangle_\Sigma - \frac{\varepsilon}{2}\|\overline{g}_\delta\|_{L^2(\Sigma)}^2$. This implies that
		\begin{align*}
			&\F_\varepsilon^*(y+\delta y) - \F_\varepsilon^*(y) - \left\langle P_{[g_a,g_b]}\left( \frac{1}{\varepsilon}y \right), \delta y \right\rangle_\Sigma\\ & = \langle y+\delta y,\overline{g}_\delta\rangle_\Sigma - \frac{\varepsilon}{2}\|\overline{g}_\delta\|_{L^2(\Sigma)}^2 - \left( \langle y,\overline{g}\rangle_\Sigma - \frac{\varepsilon}{2}\|\overline{g}\|_{L^2(\Sigma)}^2 \right) - \left\langle \overline{g}, \delta y \right\rangle_\Sigma\\
			& \le \langle y+\delta y,\overline{g}_\delta\rangle_\Sigma - \frac{\varepsilon}{2}\|\overline{g}_\delta\|_{L^2(\Sigma)}^2 - \left( \langle y,\overline{g}_\delta\rangle_\Sigma - \frac{\varepsilon}{2}\|\overline{g}_\delta\|_{L^2(\Sigma)}^2 \right) - \left\langle \overline{g}, \delta y \right\rangle_\Sigma\\
			& = \left\langle \delta y, \overline{g}_\delta - \overline{g} \right\rangle_\Sigma \le \|\delta y\|_{L^2(\Sigma)}\|\overline{g}_\delta - \overline{g}\|_{L^2(\Sigma)}.
		\end{align*}
		Since the projection operator is nonexpansive \cite[Lemma 1.10]{hinze2009}, we further get
		\begin{align*}
			\F_\varepsilon^*(y+\delta y) - \F_\varepsilon^*(y) - \left\langle P_{[g_a,g_b]}\left( \frac{1}{\varepsilon}y \right), \delta y \right\rangle_\Sigma \le \frac{1}{\varepsilon}\|\delta y\|_{L^2(\Sigma)}^2.
		\end{align*}
		Using similar arguments, we can get a lower bound for the term on the left-hand side in the form $-\frac{1}{\varepsilon}\|\delta y\|_{L^2(\Sigma)}^2$. This implies
		\begin{align*}
			\frac{1}{\|\delta y\|_{L^2(\Sigma)}}\left|\F_\varepsilon^*(y+\delta y) - \F_\varepsilon^*(y) - \left\langle P_{[g_a,g_b]}\left( \frac{1}{\varepsilon}y \right), \delta y \right\rangle_\Sigma \right| \to 0\text{ as }\|\delta y\|_{L^2(\Sigma)}\to 0.
		\end{align*}
		
	\end{proof}
	
	Let us now present the differentiability of the objective functional $J_\varepsilon$.
	\begin{theorem}\label{theorem:shapederom}
		Suppose that \Cref{assumption1} holds, and let $\varphi\in\Om(\Omega)$. The functional $J_\varepsilon$ is G{\^a}teaux differentiable at $\varphi$. The derivative in the direction $\delta\varphi\in\Om(\Omega)$ is denoted and expressed as
		\begin{align}\label{shapederivative}
			\begin{aligned}
				d_\varphi J_\varepsilon(\varphi)\delta\varphi =&\, \int_{\Omega(\varphi)} \left[ \frac{1}{2} \chi_\omega |\mathcal{S}_{1}(\Omega(\varphi)) + u_{g_r} - u_d |^2 + f \widehat{p} + \chi_\omega(\mathcal{S}_{1}(\Omega(\varphi)) + u_{g_r} - u_d)\widehat{q}  \right]\mathrm{div}(\delta\varphi\circ \tau_\varphi^{-1}) \du x\\
				& + \int_{\Omega(\varphi)} \widehat{A}\nabla \mathcal{S}_{1}(\Omega(\varphi))\cdot \nabla\widehat{p} \du x + \int_{\Omega(\varphi)} \widehat{A} \nabla\mathcal{S}_{2}(\Omega(\varphi))\cdot \nabla\widehat{q} \du x + \int_{\Omega(\varphi)}  \nabla f\cdot(\delta\varphi\circ \tau_\varphi^{-1})\widehat{p} \du x,
			\end{aligned}
		\end{align}
		where $(\widehat{p},\widehat{q}) \in H^1_0(\Omega(\varphi))\times H^{1}(\Omega(\varphi))$ is the solution for the system
		\begin{align}\label{adjoint:lagrangian}
			\begin{aligned}
				-\Delta\widehat{p} - \chi_\omega \widehat{q} & = \chi_\omega(\mathcal{S}_1(\Omega(\varphi)) + u_{g_r} - u_d) &&\text{in }\Omega(\varphi),\\
				-\Delta\widehat{q} & = 0 &&\text{in }\Omega(\varphi),\\
				\widehat{p} = \widehat{q} & = 0  &&\text{on }\Gamma(\varphi),\\
				\widehat{p} = 0,\, \widehat{q} & = - P_{[g_a,g_b]}\left(\frac{1}{\varepsilon}[ \partial_\nu \mathcal{S}_{\Omega,2}(\circphi) - \partial_\nu \mathcal{S}_2(\Omega(\varphi))]\right)&&\text{on }\Sigma,
			\end{aligned}
		\end{align}
		$(\mathcal{S}_1(\Omega(\varphi)), \mathcal{S}_2(\Omega(\varphi))) = \mathcal{S}(\Omega(\varphi)) \in H_0^1(\Omega)^2\cap H^2(\Omega)^2$, and $\widehat{A} = [ D(\delta\varphi\circ\tau_\varphi^{-1}) + (D(\delta\varphi\circ\tau_\varphi^{-1}))^{\top} - I \mathrm{div}(\delta\varphi\circ\tau_\varphi^{-1})] $.
	\end{theorem}

	\begin{proof}
		The differentiability of the objective functional follows from the differentiability of the deformation-to-state operator and the Fenchel tranform. By the chain rule we thus have
		\begin{align*}
			\begin{aligned}
				d_\varphi J_\varepsilon(\varphi)\delta\varphi =&\, \frac{1}{2}\int_\Omega \chi_\omega |\tilde{u} + \hu_{g_r}(\varphi) - u_d |^2 b'(\varphi)\delta\varphi \du x + \int_\Omega \chi_\omega b(\varphi) (\tilde{u} + \hu_{g_r}(\varphi) - u_d )(z + \hu_{g_r}'(\varphi)\delta\varphi)\du x\\  &+  \int_\Sigma P_{[g_a,g_b]}\left(\frac{1}{\varepsilon}[ \partial_\nu \mathcal{S}_{\Omega,2}(\circphi) - \partial_{\nu}w]\right)\partial_\nu v \du s\\
				=&\, \frac{1}{2}\int_\Omega \chi_\omega |\tilde{u} + \hu_{g_r}(\varphi) - u_d |^2 b'(\varphi)\delta\varphi \du x + \int_\Omega \chi_\omega b(\varphi) (\tilde{u} + \hu_{g_r}(\varphi) - u_d )(z + \hu_{g_r}'(\varphi)\delta\varphi)\du x\\
				& -\int_\Omega A(\varphi)\nabla v \cdot \nabla \overline{q}\du x - \int_\Omega \dive(A(\varphi)\nabla v)\overline{q} \du x,
			\end{aligned}
		\end{align*}
		where $(\tilde{u},w) = \mathcal{S}(\varphi)$, $(z,v) = \mathcal{S}'(\varphi)\delta\varphi$. We note that to achieve the second equality, we used the fact that $D\varphi_{|\Sigma} = 0$, which implies that $A(\varphi) = I$ on $\Sigma$, and thus from \eqref{qbar} we infer that
		\begin{align*}
			\int_\Sigma P_{[g_a,g_b]}\left(\frac{1}{\varepsilon}[ \partial_\nu \mathcal{S}_{\Omega,2}(\circphi) - \partial_{\nu}w]\right)\partial_\nu v \du s = -  \int_\Sigma \overline{q} (A(\varphi)\nabla v)\cdot\nu \\
			= -\int_\Omega A(\varphi)\nabla v \cdot \nabla \overline{q}\du x - \int_\Omega \dive(A(\varphi)\nabla v)\overline{q} \du x.
		\end{align*}
		From the adjoint system \eqref{adjoint:eulerian}, we further get
		\begin{align}\label{der}
			\begin{aligned}
				d_\varphi J_\varepsilon(\varphi)\delta\varphi = &\, \frac{1}{2}\int_\Omega \chi_\omega |\tilde{u} + \hu_{g_r}(\varphi) - u_d |^2 b'(\varphi)\delta\varphi \du x + \langle d_{(\tilde{u},w)}E_\Omega[\varphi,\mathcal{S}_\Omega(\varphi)](z + \hu_{g_r}'(\varphi)\delta\varphi,v), (p,q) \rangle_{H_0^1(\Omega)^2}\\
				&-\int_\Omega \dive(A(\varphi)\nabla v)\overline{q} \du x - \int_\Omega b(\varphi)(z + \hu_{g_r}'(\varphi)\delta\varphi)\overline{q}\chi_\omega\du x.
			\end{aligned}
		\end{align}
		From \eqref{Eder2}, integration by parts, and since $(p,q)\in H_0^1(\Omega)^2$ we see that
		\begin{align*}
			\langle & d_{(\tilde{u},w)}E_\Omega[\varphi,\mathcal{S}_\Omega(\varphi)](z + \hu_{g_r}'(\varphi)\delta\varphi,v), (p,q) \rangle_{H_0^1(\Omega)^2}\\ & = - \int_\Omega \dive(A(\varphi)\nabla( z + \hu_{g_r}'(\varphi)\delta\varphi))p \du x  - \int_\Omega \dive(A(\varphi)\nabla v){q} \du x
			- \int_\Omega b(\varphi)( z + \hu_{g_r}'(\varphi)\delta\varphi) {q}\chi_\omega\du x.
		\end{align*}
		Plugging this into equation \eqref{der} and using \eqref{system:linear} will give us
		\begin{align*}
			d_\varphi J_\varepsilon(\varphi)&\delta\varphi =  \frac{1}{2}\int_\Omega \chi_\omega |\tilde{u} + \hu_{g_r}(\varphi) - u_d |^2 b'(\varphi)\delta\varphi \du x - \int_\Omega \dive(A(\varphi)\nabla( z + \hu_{g_r}'(\varphi)\delta\varphi))p \du x   \\
			&- \int_\Omega \dive(A(\varphi)\nabla v)({q}+\overline{q}) \du x - \int_\Omega b(\varphi)( z + \hu_{g_r}'(\varphi)\delta\varphi) ({q}+\overline{q})\chi_\omega\du x\\
			=&\, \frac{1}{2}\int_\Omega \chi_\omega |\tilde{u} + \hu_{g_r}(\varphi) - u_d |^2 b'(\varphi)\delta\varphi \du x + \int_\Omega b'(\varphi)\delta\varphi \hf(\varphi) p \du x + \int_\Omega b(\varphi)\hf'(\varphi)\delta\varphi  p \du x\\
			& + \int_\Omega \dive(A'(\varphi)\delta\varphi \nabla(\tilde{u} + \hu_{g_r}(\varphi)))p \du x + \int_\Omega \dive(A'(\varphi)\delta\varphi \nabla v)(q+\overline{q}) \du x\\
			&+ \int_\Omega b'(\varphi)\delta\varphi (\tilde{u} + \hu_{g_r}(\varphi) - u_d)(q+\overline{q})\chi_\omega \du x.
		\end{align*}
		Another integration by parts on the terms on the second line of the right-hand side above, together with the fact $p = q + \overline{q} = 0$ on $\Gamma$ and $D\delta \varphi =\mathbf{0}$ on $\Sigma$, lead to
		\begin{align}\label{prepder}
			\begin{aligned}
				d_\varphi J_\varepsilon(\varphi)&\delta\varphi = \frac{1}{2}\int_\Omega \chi_\omega |\tilde{u} + \hu_{g_r}(\varphi) - u_d |^2 b'(\varphi)\delta\varphi \du x + \int_\Omega b'(\varphi)\delta\varphi \hf(\varphi) p \du x + \int_\Omega b(\varphi)\hf'(\varphi)\delta\varphi  p \du x\\
				& - \int_\Omega A'(\varphi)\delta\varphi \nabla(\tilde{u} + \hu_{g_r}(\varphi))\cdot\nabla p \du x - \int_\Omega A'(\varphi)\delta\varphi \nabla w\cdot \nabla(q+\overline{q}) \du x\\
				&+ \int_\Omega b'(\varphi)\delta\varphi (\tilde{u} + \hu_{g_r}(\varphi) - u_d)(q+\overline{q})\chi_\omega \du x.
			\end{aligned}
		\end{align}
		
		Lastly, for any $u_1,u_2\in H^1(\Omega)$, we use the following identities:
		\begin{align*}
			&\int_\Omega b'(\varphi)\delta\varphi\, u_1 \du x = \int_{\Omega(\varphi)} (u_1\circ\tau_\varphi^{-1} )\dive(\delta\varphi\circ\tau_{\varphi}^{-1}) \du x ; \quad \text{and}\\
			&\int_\Omega A'(\Omega)\delta\varphi \nabla u_1 \cdot \nabla u_2 \du x = - \int_{\Omega(\varphi)} \widehat{A} \nabla (u_1\circ\tau_\varphi^{-1}) \cdot \nabla (u_2\circ\tau_\varphi^{-1}) \du x,
		\end{align*}
		to facilitate the change of the integrals from $\Omega$ to $\Omega(\varphi)$ in \eqref{prepder} which gives us \eqref{shapederivative}, with $(\widehat{p},\widehat{q})=(p,q+\overline{q})\circ \tau_\varphi^{-1}$. {By a change of variables we transform the integral from $\Omega$ to $\Omega(\varphi)$ in the adjoint system \eqref{adjoint:eulerian} and deduce that $(\hat p,\hat q)$ solves\eqref{adjoint:lagrangian}.}
	\end{proof}

	\noindent\textbf{Zol{\'e}sio-Hadamard structure.} We write the derivative of the objective function as a boundary integral on the variable boundary $\Gamma(\varphi)$. In this way, we can get the expression for the shape gradient, which is often referred to as the Zol{\'e}sio-Hadamard structure \cite[page 479]{delfour2011}.
	
	\begin{corollary}\label{corzolhad}
		Suppose that the assumptions in \Cref{theorem:shapederom} hold. Then the derivative of $J_\varepsilon$ at $\varphi\in\Om(\Omega)$ in the direction $\delta\varphi\in\Om(\Omega)$ can be expressed as
		\begin{align}\label{zolhad}
			d_\varphi J_\varepsilon(\varphi)\delta\varphi = \int_{\Gamma(\varphi)} \left[ \frac{1}{2}\chi_\omega |u(\varphi) - u_d |^2 + \nabla u(\varphi) \cdot\nabla\widehat{p} + \nabla w(\varphi) \cdot\nabla\widehat{q}\, \right](\delta\varphi\circ \tau_\varphi^{-1}\cdot\nu)\du s.
		\end{align}
		where $(u(\varphi),w(\varphi)) = (\mathcal{S}_1(\Omega(\varphi)) + u_{g_r}, \mathcal{S}_2(\Omega(\varphi)))$.
	\end{corollary}
	
	\begin{proof}
		We note that for any $u,v\in H^1(\Omega(\varphi))$ with $u = v = 0$ on $\Gamma(\varphi)$
		\begin{align*}
			\int_{\Omega(\varphi)} \widehat{A}\nabla u\cdot \nabla v \du x = \int_{\Gamma(\varphi)}  (\nabla u\cdot\nabla v)[\theta  \cdot \nu ]\du s-\int_{\Omega(\varphi)} \Delta u[ \theta \cdot \nabla v ] + \Delta v[ \theta \cdot \nabla u ] \du x,
		\end{align*}
		where we used the notation $\theta := \delta\varphi\circ\tau_\varphi^{-1}$, and $\Delta u,\Delta v \in L^2(\mathbb{R}^d)$ in the sense of distributions (see e.g. \cite[pages 487--488]{delfour2011}). From the identity above and using the notation $(u(\varphi),w(\varphi)) = (\mathcal{S}_1(\Omega(\varphi)) + u_{g_r}, \mathcal{S}_2(\Omega(\varphi)))$ we get
		\begin{align*}
			d_\varphi J_\varepsilon(\varphi)\delta\varphi & = \int_{\Omega(\varphi)} \left[ \frac{1}{2} \chi_\omega |u(\varphi) - u_d |^2 + f \widehat{p} + \chi_\omega(u(\varphi) - u_d)\widehat{q}  \right]\mathrm{div}\theta \du x + \int_{\Omega(\varphi)}  (\nabla f\cdot\theta)\widehat{p} \du x\\
			&- \int_{\Omega(\varphi)} \Delta u(\varphi) [ \theta \cdot \nabla \widehat{p} ] + \Delta \widehat{p}[ \theta \cdot \nabla u(\varphi)  ] \du x + \int_{\Gamma(\varphi)} (\nabla u(\varphi) \cdot\nabla\widehat{p})[\theta  \cdot \nu ]\\
			&- \int_{\Omega(\varphi)} \Delta w(\varphi)[ \theta \cdot \nabla \widehat{q} ] + \Delta \widehat{q}[ \theta \cdot \nabla w(\varphi) ] \du x + \int_{\Gamma(\varphi)} (\nabla w(\varphi)\cdot\nabla\widehat{q})[\theta  \cdot \nu ]\\
			& = \int_{\Omega(\varphi)} \left[ \frac{1}{2} \chi_\omega |u(\varphi) - u_d |^2 + f \widehat{p} + \chi_\omega(u(\varphi)  - u_d)\widehat{q}  \right]\mathrm{div}\theta \du x + \int_{\Omega(\varphi)}  (\nabla f\cdot\theta )\widehat{p} \du x\\
			& + \int_{\Omega(\varphi)} f[ \theta \cdot \nabla \widehat{p} ] + (\chi_\omega \widehat{q} + \chi_\omega(u(\varphi) - u_d))[ \theta \cdot \nabla u(\varphi)  ] \du x + \int_{\Gamma(\varphi)} (\nabla u(\varphi) \cdot\nabla\widehat{p})[\theta  \cdot \nu ]\\
			&+ \int_{\Omega(\varphi)} \chi_\omega(u(\varphi)  - u_d)[ \theta \cdot \nabla \widehat{q} ]  \du x + \int_{\Gamma(\varphi)} (\nabla w(\varphi)\cdot\nabla\widehat{q})[\theta  \cdot \nu ].
		\end{align*}
		To simplify even further, we have by integration by parts:
		\begin{align*}
			&\bullet\ \int_{\Omega(\varphi)}  \chi_\omega |u(\varphi) - u_d |^2 \mathrm{div}\theta \du x = \int_{\partial\Omega(\varphi)} \chi_\omega |u(\varphi) - u_d |^2 (\theta\cdot\nu) - 2\int_{\Omega(\varphi)} \chi_\omega (u(\varphi) - u_d)(\theta\cdot\nabla u(\varphi)),\\
			&\bullet\ \int_{\Omega(\varphi)}  f\widehat{p}\, \mathrm{div}\theta \du x = -\int_{\Omega(\varphi)}  (\nabla f\cdot\theta )\widehat{p} \du x - \int_{\Omega(\varphi)} f(\theta\cdot\nabla \widehat{p}\,),\\
			&\bullet\ \int_{\Omega(\varphi)}  \chi_\omega(u(\varphi)  - u_d)\widehat{q}\, \mathrm{div}\theta \du x = - \int_{\Omega(\varphi)} \chi_\omega(u(\varphi)  - u_d)(\theta\cdot\nabla \widehat{q}\,) - \int_{\Omega(\varphi)} \chi_\omega\widehat{q}(\theta\cdot\nabla u(\varphi)\,).
		\end{align*}
		Therefore, we get \eqref{zolhad}
	\end{proof}
	
	We end this subsection with the following classical first-order necessary condition.
	\begin{proposition}[First Order Necessary Condition]
		Let $\varphi^*\in \Om(\Omega)$ be a local minimizer of $J_\varepsilon$. Then $d_\varphi J_\varepsilon(\varphi^*)( \varphi - \varphi^*) \ge   0$ for any $\varphi \in  \Om(\Omega)$.
	\end{proposition}

	\section{Numerical resolution} \label{section:numerics}

	In this section, we discuss the numerical implementation of the low-regret problem \eqref{low-regret}. We are especially interested  in the effect of missing information of the Dirichlet data on  the nature of the solutions, or at least on the optimal  values of the cost-functional.  We also numerically investigate the effect of the parameter $\varepsilon>0$ in the  low-regret formulation on the numerical solutions.
	
	Let us recall from \Cref{corzolhad} that the derivative can be written in the form $$d_\varphi J_\varepsilon(\varphi)\delta\varphi = \displaystyle\int_{\Gamma(\varphi)}\nabla J_\varepsilon(\varphi)\nu \cdot (\delta\varphi\circ \tau_\varphi^{-1}) \du s$$
	 with $\nabla J_\varepsilon(\varphi) =  \frac{1}{2}\chi_\omega |u(\varphi) - u_d |^2 + \nabla u(\varphi) \cdot\nabla\widehat{p} + \nabla w(\varphi) \cdot\nabla\widehat{q}$. 
	 
	 The intuitive choice for a descent direction would then be $\delta\varphi = -(\nabla J_\varepsilon(\varphi)\nu)\circ\tau_\varphi$ on $\Gamma(\varphi)$ so that $$d_\varphi J_\varepsilon(\varphi)\delta\varphi = - \|\nabla J_\varepsilon(\varphi)\nu\|_{L^2(\Gamma(\varphi))^d}^2 \le 0.$$
	We employ the so-called traction method {\cite{azegami1994,shimoda1995,shimoda1997}} that smoothly extends $-\nabla J_\varepsilon(\varphi)\nu$ over $\Omega(\varphi)$. For this purpose, we compute the solution $\bG(\varphi):\Omega(\varphi)\to \mathbb{R}^d$ of the system
	\begin{align}\label{traction}
		\left\{
		\begin{aligned}
			\alpha\Delta\bG &= 0 &&\text{in }\Omega(\varphi),\\
			\alpha\partial_\nu\bG + \bG &= -\nabla J_\varepsilon(\varphi)\nu &&\text{on }\Gamma(\varphi),\\
			\bG & = 0 &&\text{on }\Sigma,
		\end{aligned}
		\right.
	\end{align}
	where $\alpha>0$ is sufficiently small. We can then choose $\delta\varphi = \bG\circ\tau_\varphi$, so that $\delta\varphi_{|\Gamma(\varphi)} \approx -(\nabla J_\varepsilon(\varphi)\nu)\circ\tau_\varphi $. {We note that if $f\in H^1(D)$ and $g_r\in H^{\frac{5}{2}}(\Sigma)$ then $\bG\in W^{3,\infty}(D;\mathbb{R}^d)$ and consequently $\delta\varphi = \bG\circ\tau_\varphi\in W^{3,\infty}(D;\mathbb{R}^d)$, as required by $\mathcal{O}(\Omega)$. In the numerical realisation we shall use P1-elements so that this regularity will not be supported. We expect that this does not have a qualitatively strong effect on the solution.}
	
	An advantage of having the shape parametrized by the deformation field is that it allows us to use numerical methods that have been successfully utilized for optimal control problems.  We look for minimizing deformation field iteratively, beginning with $\varphi^0 = 0$, and generate the minimizing sequence $(\varphi^k)$ by $\varphi^{k+1} = \varphi^{k} + t^{k} \delta\varphi^{k}$, where $\delta\varphi^{k} = \bG(\varphi^{k})\circ\tau_{\varphi^{k}}$. The step size $t^{k-1}$ is chosen by a modified Barzilai--Borwein method \cite{barzilai1988}, i.e.,
	\begin{align}\label{BarBorstep}
		t^{k}  = \sigma^k\times \left\{
		\begin{aligned}
			\frac{\int_{\Gamma} (\varphi^{k-1} - \varphi^{k-2})\cdot (\delta\varphi^{k-1} - \delta\varphi^{k-2})}{\|\delta\varphi^{k-1} - \delta\varphi^{k-2} \|_{L^2(\Gamma)}^2} &&\text{if }k\text{ is odd},\\
			\frac{\|\varphi^{k-1} - \varphi^{k-2} \|_{L^2(\Gamma)}^2}{\int_{\Gamma} (\varphi^{k-1} - \varphi^{k-2})\cdot (\delta\varphi^{k-1} - \delta\varphi^{k-2})} &&\text{if }k\text{ is even},\\
		\end{aligned}
		\right.
	\end{align}
	where $\sigma\in (0,1)$ is a fixed scaling parameter.
	The power in the scaling parameter is added so as to make sure that the deformed domain $\Omega(\varphi^{k})$ does note have a boundary $\Gamma(\varphi^{k})$ that crosses itself.
	Since the computation of $t^k$ requires two previous steps, we  need to also define $\varphi^{1}$ before entering the iterative computation of the solution.

	Summarizing, we have the following algorithm:
	\begin{algorithm}
		\caption{Gradient method with a weighted Barzilai--Borwein step size}\label{alg:cap}
		\begin{algorithmic}
			\State\textbf{Initialize:} $0<\texttt{tol}<<1$, $\texttt{test} = 1$, $0< \sigma < 1$
			\State \qquad\qquad $\varphi^{0} = 0$, $\delta\varphi^0 = \bG(\varphi^0)$, $\varphi^1 \gets \varphi^0 + 0.5\times \delta\varphi^0$, $\delta\varphi^1 \gets \bG(\varphi^1)\circ\tau_{\varphi^1}$ and set $k = 2$
			\While{$\texttt{test} \ge \texttt{tol}$}
			\State identify $t^k$ by \eqref{BarBorstep}
			\State $\varphi^{k+1} \gets \varphi^{k} + t^{k} \delta\varphi^{k}$ and compute $\delta\varphi^{k+1} \gets \bG(\varphi^{k+1})\circ\tau_{\varphi^{k+1}}$
			\State $\texttt{test}\gets \max\{|J_\varepsilon(\varphi^{k+1}) - J_\varepsilon(\varphi^{k})|, \|\delta\varphi^{k+1}\|_{L^2(\Gamma)}\}$
			\State $k \gets k+1$
			
			\EndWhile
		\end{algorithmic}
	\end{algorithm}

	\subsection{Initial set-up} To find the minimizing deformation field, we are tasked to solve three equations (the state equations \eqref{system:couple}, adjoint equations \eqref{adjoint:lagrangian}, and the deformation field equations \eqref{traction}). We use the finite element method --- through the open-source software \texttt{FreeFem++} \cite{hecht2012} --- to approximate the solutions of these equations using $\mathbb{P}^1$ finite elements.

	The initial domain $\Omega\subset D$ is defined as the annular region that is bounded by $\Sigma = \{ (x_1,x_2)\in\mathbb{R}^2 : x_1^2 + x_2^2 = 4 \}$ and $\Gamma = \{ (x_1,x_2)\in\mathbb{R}^2 : x_1^2 + x_2^2 = (3/4)^2 \} $. The subdomain $\omega \subset \Omega$ is defined as $\omega = \{ (x_1,x_2)\in \mathbb{R}^2 : 1 \le x_1^2 + x_2^2 \le (7/4)^2 \}$.

	Throughout the implementation, we shall use two profiles for the function $u_d :\omega\to \mathbb{R}$. Namely, $u_d$ is the restriction to $\omega$ of the solution of the problem
	\begin{align}\label{targetdef}
		\left\{
		\begin{aligned}
			-\Delta v & = f \text{ in }\Omega_d,\\
			v & = 0  \text{ on }\Gamma_d,\\
			v & = g_r \text{ on }\Sigma,
		\end{aligned}
		\right.
	\end{align}
	where $f\in L^2(D)$ is the source function in \eqref{system:couple}, $g_r\in H^{\frac{1}{2}}(\Sigma)$ is a given Dirichlet-data and $\Omega_d$ is an annular domain whose exterior boundary is the same $\Sigma$ defined above, and its interior boundary $\Gamma_d$ is one of the following cases:
	\begin{align*}
		&\Gamma_d = \left\{(x_1,x_2)\in \mathbb{R}^2 : (x_1 - 0.1)^2 + x_2^2 = \frac{1}{16}\right\},\\
		&\Gamma_d = \left\{(x_1,x_2)\in \mathbb{R}^2 :  x_1 =0.4( \cos(t) + 0.4\cos(2t) ),\, x_2 =  0.3\sin(t) , t\in [0,2\pi]  \right\},
	\end{align*}
	which we respectively will refer to as $\Gamma_d$ (circle) and $\Gamma_d$ (arrow head). Throughout the remaining part of this manuscript, we use the functions $f = 1 $ and $g_r  = 0.1\cos(2\pi x_1)\sin(2\pi x_2)$. We refer to the target profile derived from $\Gamma_d$ (circle) and $\Gamma_d$ (arrow head) as $u_d$ (circle) and $u_d$ (arrow head), respectively.
	
	\begin{figure}[h]
		\centering
		\includegraphics[width=0.45\textwidth]{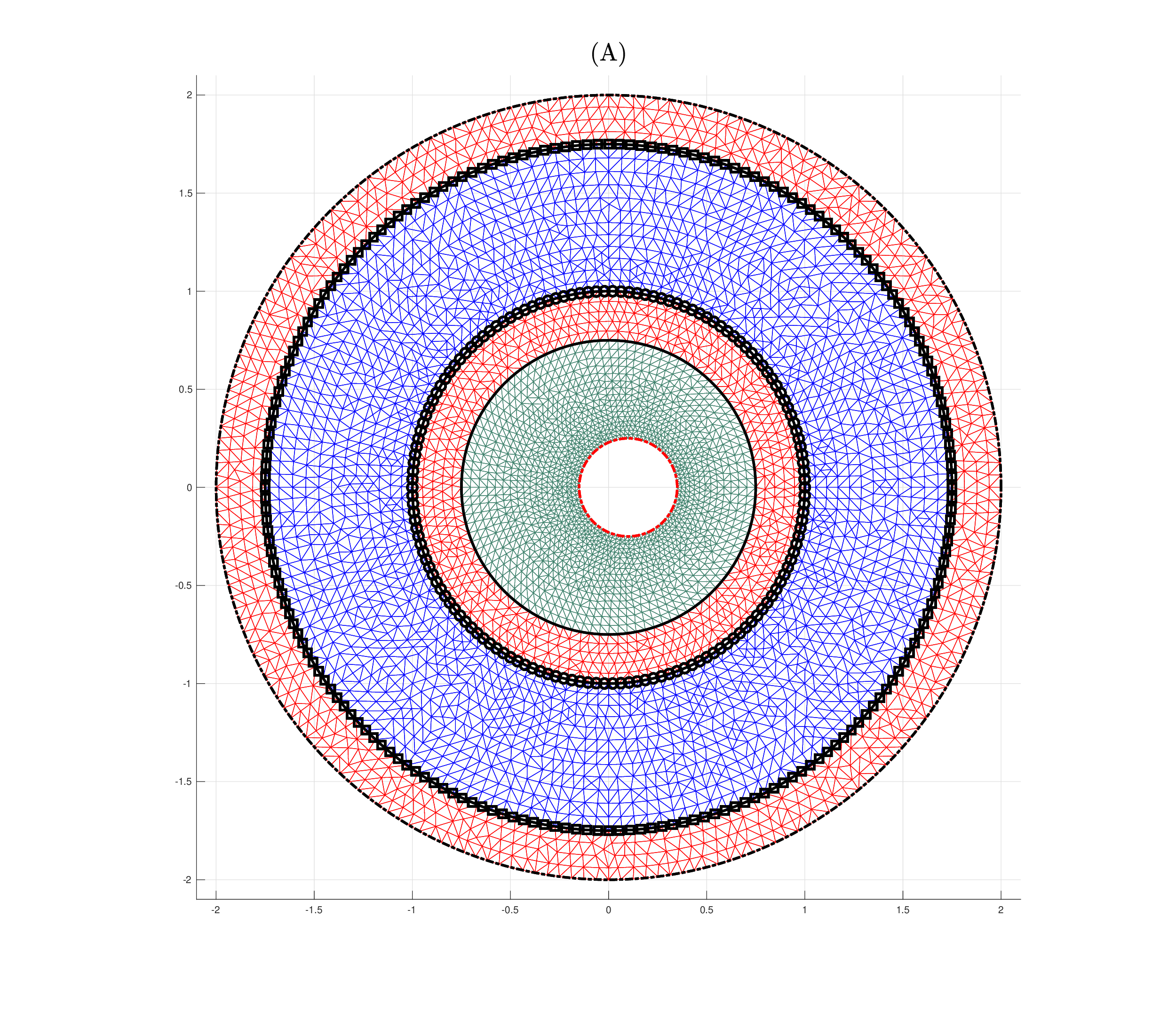}\hspace{-.4in} \includegraphics[width=0.56\textwidth]{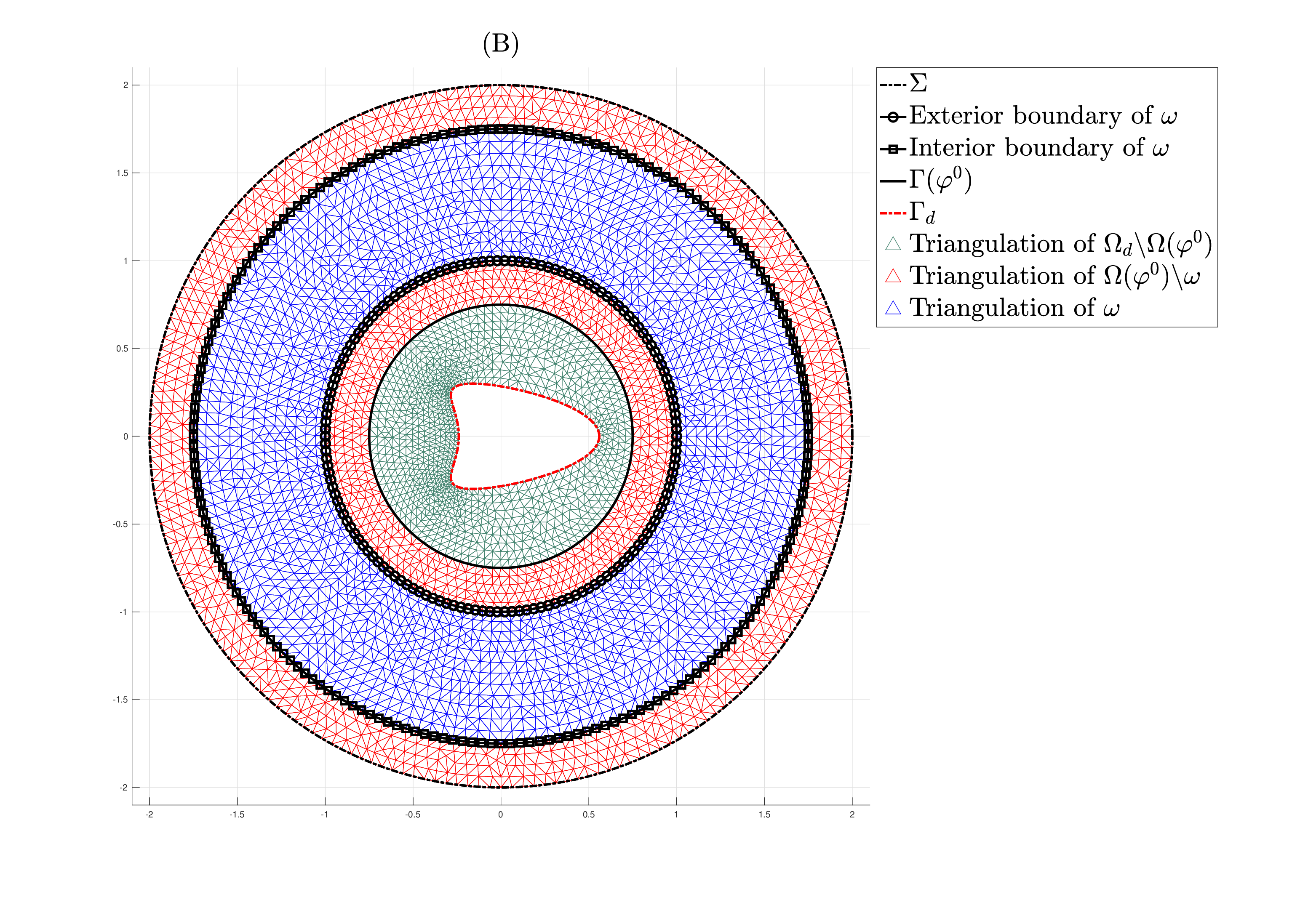}
		\caption{Initial set-up of the boundaries $\Sigma$, $\Gamma(\varphi^0)$ and $\Gamma_d$ and the triangulation of the domains $\omega$, $\Omega(\varphi^0)$ and $\Omega_d$ with $\Gamma_d$ (circle) (A) and $\Gamma_d$ (arrow head) (B).}
		\label{figure:domssetup}
	\end{figure}
	
	The discretization of $\Gamma_d$ is done by dividing the interval $[0,2\pi]$ by $n = 80$ points. The boundary $\Gamma(\varphi^0)$ is then discretized by dividing the interval $[0,2\pi]$ with $n = 100$ points and use its parametrization. Similarly, the interior and exterior boundaries of $\omega$, and $\Sigma$ are discretized using their polar coordinate parametrization and dividing the interval $[0,2\pi]$ respectively with $n = 120$, $n = 140$ and $n = 160$ points. Each region bounded by two consecutive boundaries is then triangulated via Delaunay. See \Cref{figure:domssetup} for visualization, where the initial domain $\Omega(\varphi^0)$ is the region with the red and blue triangulations, while $\Omega_d$ is the domain consisting of the the regions with red, blue and green triangulations. To visualize the behavior of the deformation fields, we shall only show the plots of the deformation of the boundary $\Gamma$, i.e. if $\varphi$ is a deformation that is of interest we shall show the boundary $\Gamma(\varphi)$.
	
	For the box constraints on the missing data $g_\delta$, we use $g_a = -0.2$ and $g_b = 0.2$.
	
	\subsection{The nominal deformation} In this section, we discuss the construction of the nominal deformation field $\circphi$. Since the solutions of the low-regret optimization problem that we are looking for are supposed to correspond to missing Dirichlet data around the given Dirichlet profile $g_r$, we deemed it reasonable to define $\circphi$ in such a way that it corresponds to $g_r$. To be specific, $\circphi$ solves the optimization problem
	\begin{align}
		&\min_{\varphi\in\Om(\Omega)} \widetilde J(\varphi) = \frac{1}{2}\int_{\omega}
		\left| u - u_d \right|^2 \du x\label{objfuncnom} \\
		\text{subject to}&\nonumber\\
		&\left\{
		\begin{aligned}
			-\Delta u & = f  &&\text{in } \Omega(\varphi),\\
			u & = 0&& \text{on }\Gamma(\varphi),\\
			u & = g_r  && \text{on }\Sigma.
		\end{aligned}
		\right.\label{nominalpoisson}
	\end{align}
	The optimization problem \eqref{objfuncnom}--\eqref{nominalpoisson} is well-known and numerical techniques have been well studied, see for example \cite{dapogny2018,delfour2011,henrot2010,murai2013}. In our case, we use the the gradient method discussed above.
	
	This minimization problem also serves as a test for the performance of the Barzilai--Borwein method:  since the state equation \eqref{nominalpoisson} has the same Dirichlet profile as the Poisson equation from which we derive our target profile $u_d$, the gradient method should be able to give us a boundary $\Gamma(\circphi)$ that is close to $\Gamma_d$. We observe from \Cref{figure:nominal} that this  is indeed the case. In \Cref{figure:nominal}, blocks (A) and (C) show the boundaries $\Gamma_d$, $\Gamma(\circphi)$, $\Gamma(\varphi^0)$, and the interior boundary of $\omega$. Additionally, we plotted the points of discretization of the boundaries $\Gamma(\varphi^k)$
	that converge to $\Gamma(\circphi)$. Meanwhile, blocks (B) and (D) depict the close-up comparison between $\Gamma_d$ and $\Gamma(\circphi)$. We see in \Cref{figure:nominal} (B) that the boundary $\Gamma(\circphi)$ is close to $\Gamma_d$, where the difference can be attributed to the limited access to the data
	and/or due to machine error. On the more nuanced $\Gamma_d$ (arrow head), we observe that our method is not able to identify the left side of the arrow which exhibits non-convexity of the region inside $\Gamma_d$. This phenomena is not exclusive to our method. Indeed, there are works dedicated to identifying shapes that have non-convexity, see for example \cite{rabago2024}.
	\begin{figure}[h]
		\centering
		\includegraphics[width=0.9\textwidth]{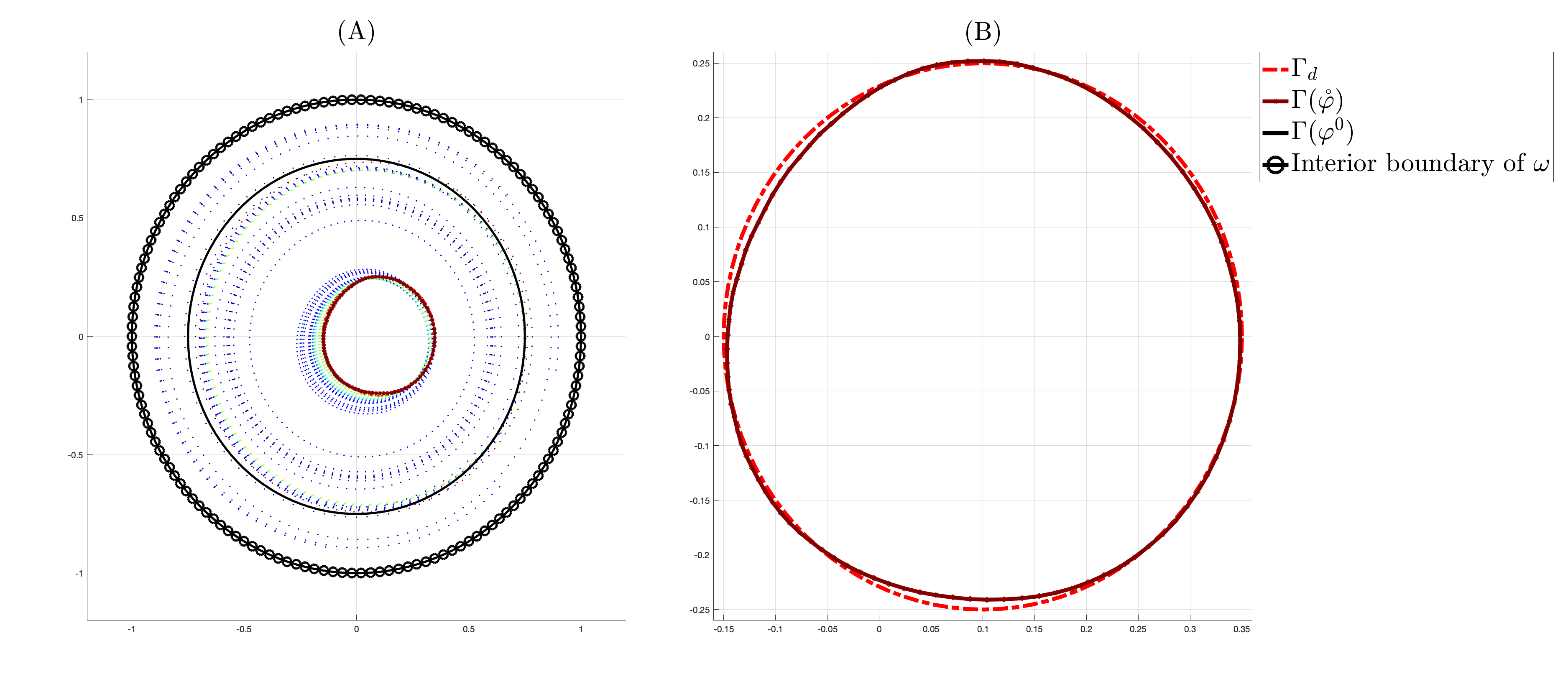}\\ \hspace{-.5in}\includegraphics[width=0.8\textwidth]{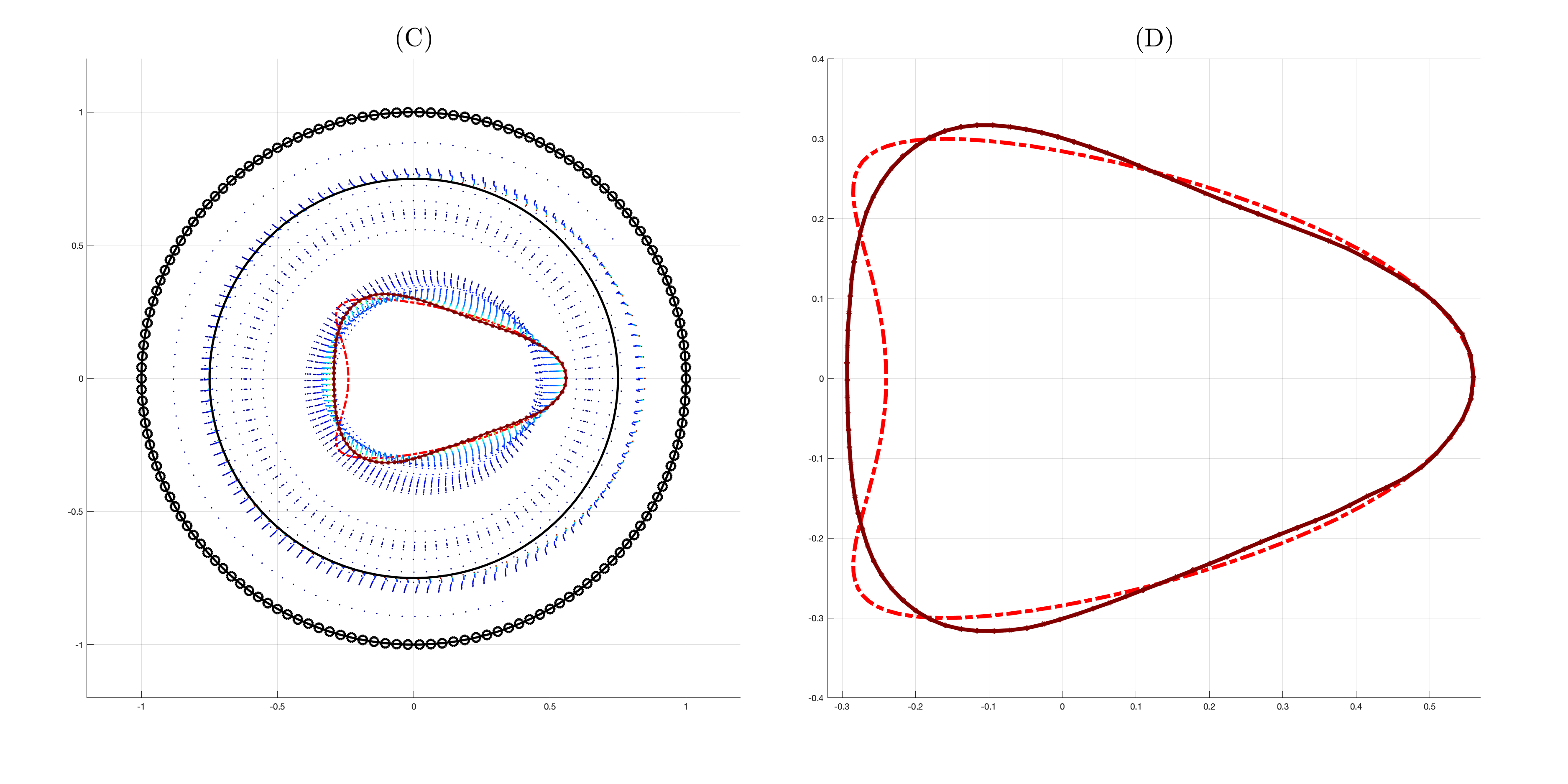}
		\caption{The evolution of the initial boundary $\Gamma(\varphi^0)$ towards $\Gamma(\circphi)$ with target profile $u_d$ (circle) (A) and $u_d$ (arrow head) (C), and the zoomed-in comparison of the boundaries $\Gamma_d$ and $\Gamma(\circphi)$ with $u_d$ (circle) (B) and $u_d$ (arrow head) (D).}
		\label{figure:nominal}
	\end{figure}

	\subsection{Low-regret solutions} 
	We now solve the low-regret problem at specific values of $\varepsilon>0$. 
	Beginning with $\varepsilon = 0.5$, \Cref{figure:eps05} (A1) and \Cref{figure:eps05} (B1) show the evolution of the initial boundary $\Gamma(\varphi^0)$ towards the solution $\Gamma(\varphi_\varepsilon)$ with $u_d$ (circle) and $u_d$ (arrow head) as target profiles, respectively. We see that the boundary $\Gamma(\varphi_\varepsilon)$ for both target profiles is considerably farther from the boundary $\Gamma_d$ than the boundary $\Gamma(\circphi)$. This is due to the fact that the low-regret problem takes into account the missing data $g_\delta$ which is compensated by the Fenchel transform.
	
	\Cref{figure:eps05} (A2) and \Cref{figure:eps05} (B2) show the progression of the objective function values at each iteration. 
	As expected from the Barzilai--Borwein method, the iterations of $J_\varepsilon$ tend to a neighborhood of zero in a non-monotone fashion.
	We observe that the value of the tracking part of the cost become small as $k\to\infty$, which behaves similarly as $J_\varepsilon$. This relates to the fact that $\|\overline{g}(\varphi_\varepsilon)\|_{L^2(\Sigma)} = 5.035\times 10^{-8}$ and $\|\overline{g}(\varphi_\varepsilon)\|_{L^2(\Sigma)} = 3.419\times 10^{-7}$ for the cases where the target profiles are $u_d$ (circle) and $u_d$ (arrow head), respectively. As a consequence, this gives small value for the evaluation of the Fenchel transform.

	\begin{figure}[h]
		\centering
		\includegraphics[width=.7\textwidth]{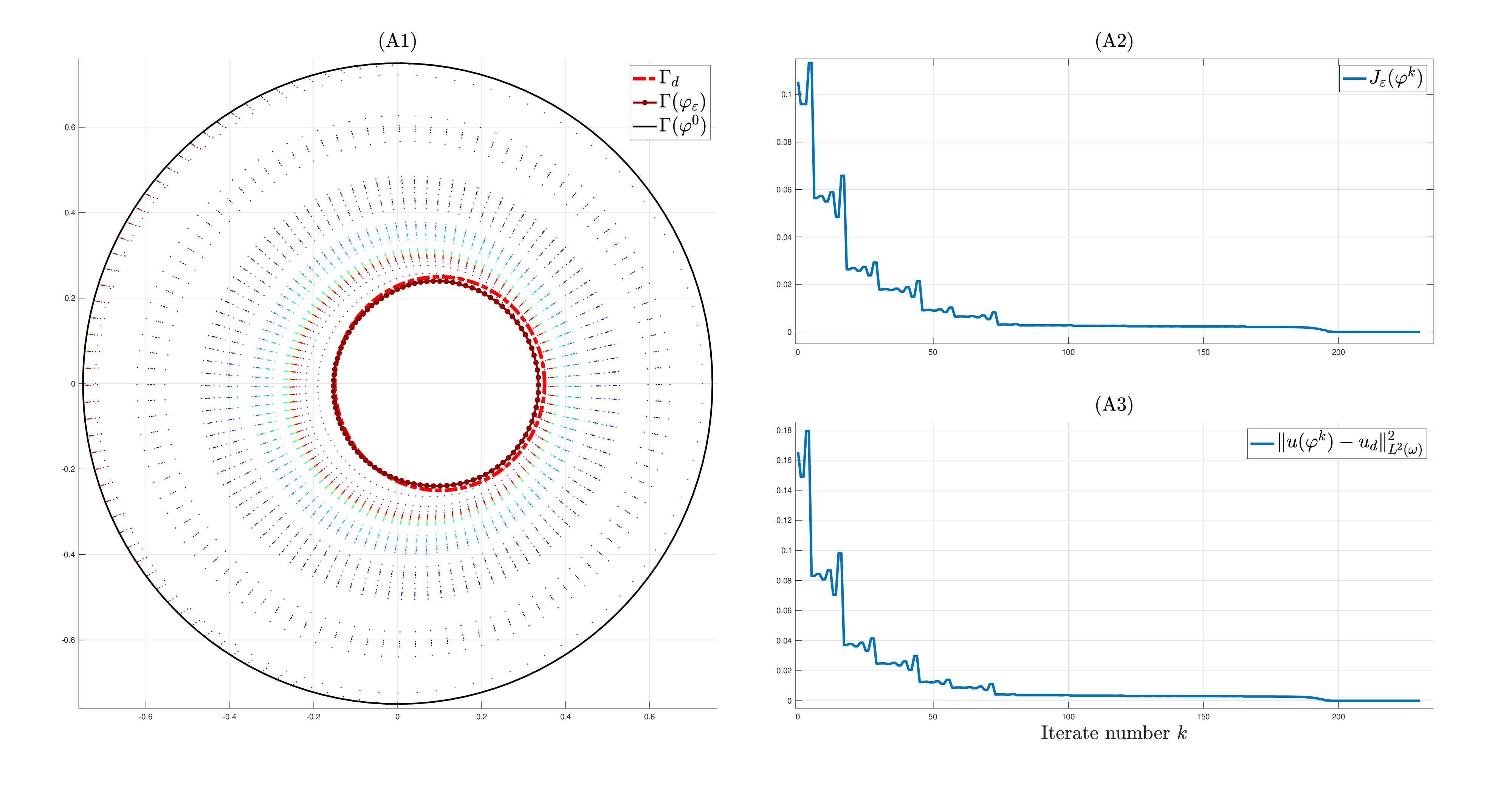}\\ \includegraphics[width=.7\textwidth]{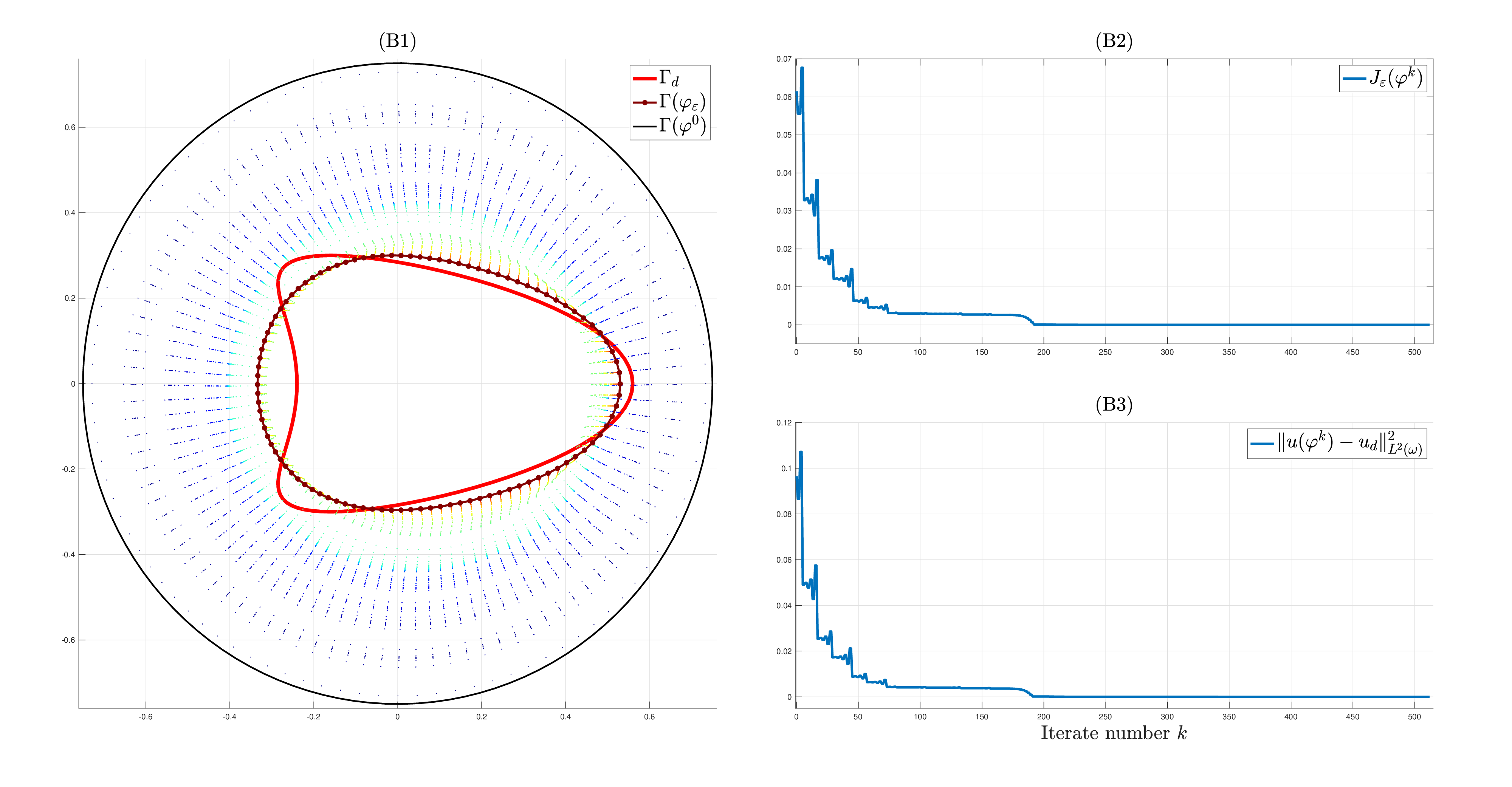}
		\caption{The evolution of the initial boundary $\Gamma(\varphi^0)$ towards $\Gamma(\varphi_\varepsilon)$ with target profile $u_d$ (circle) (A1) and $u_d$ (arrow head) (B1); the trend of the objective function value with respect to the iteratiosn with $u_d$ (circle) (A2) and $u_d$ (arrow head) (B2); and the trend of the tracking functional with respect to the iteratiosn with $u_d$ (circle) (A3) and $u_d$ (arrow head) (B3).}
		\label{figure:eps05}
	\end{figure}
	
	\begin{figure}[h]
		\centering
		\includegraphics[width=1\textwidth]{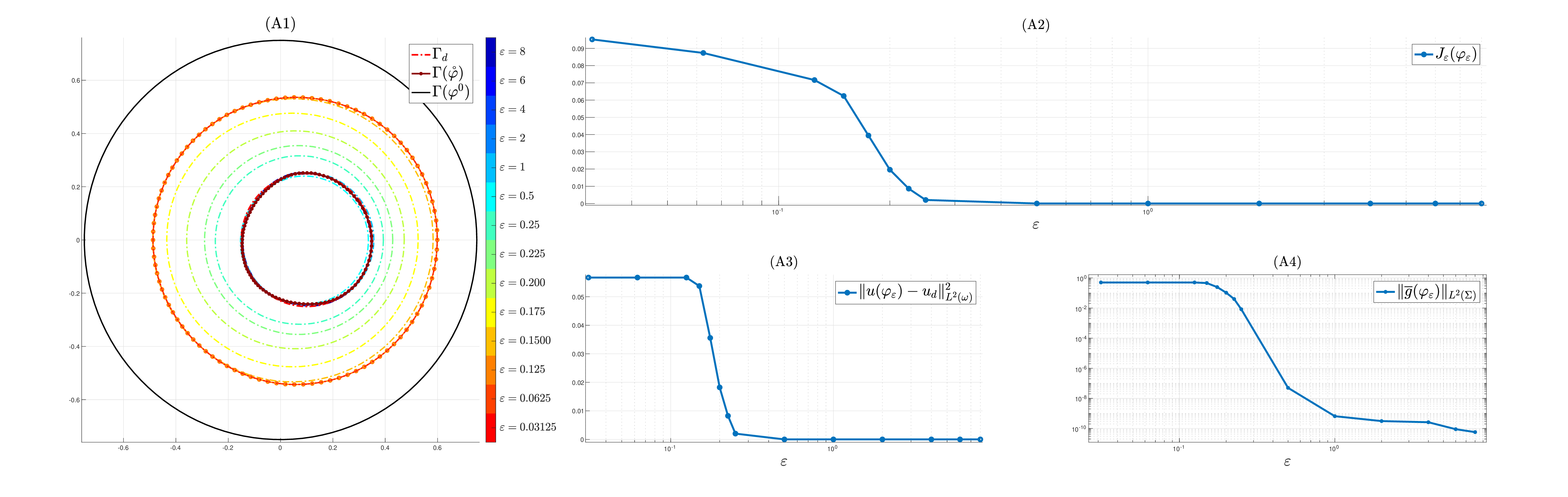}\\
		\includegraphics[width=1\textwidth]{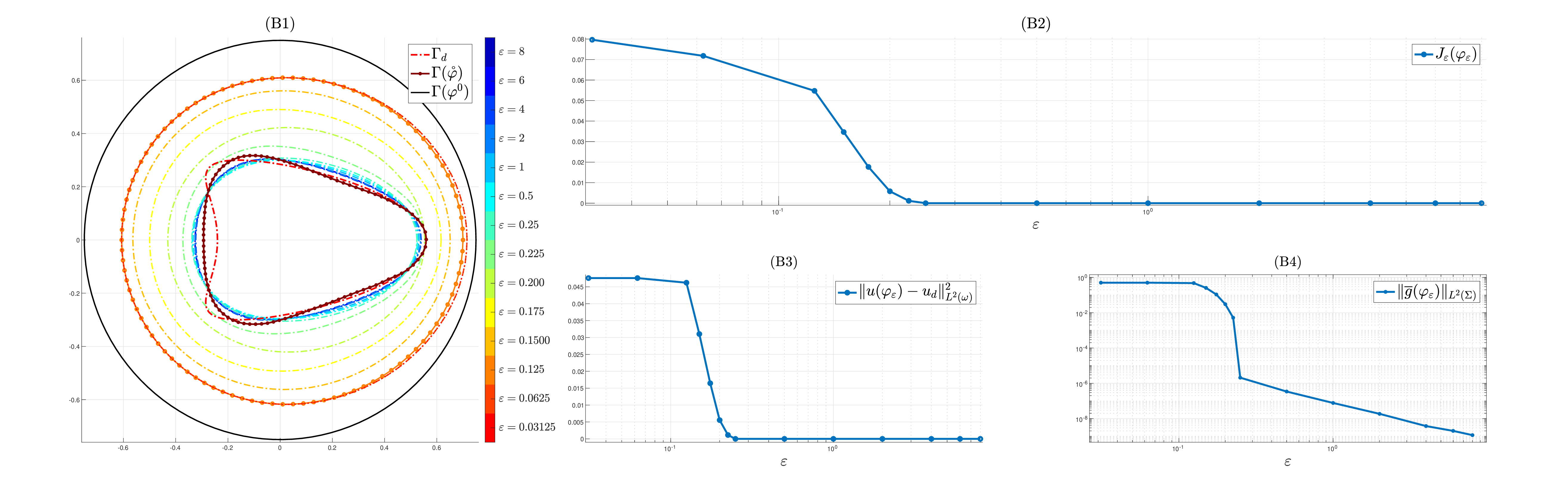}
		\caption{Comparison of the boundaries $\Gamma(\varphi_\varepsilon)$ for varying values of $\varepsilon$ with target profile $u_d$ (circle) (A1) and $u_d$ (arrow head) (B1); the optimal value of the objective function $J_\varepsilon(\varphi_\varepsilon)$ plotted against $\varepsilon$  with $u_d$ (circle) (A2) and $u_d$ (arrow head) (B2); the value of the tracking functional at the optimal deformation plotted against $\varepsilon$ with $u_d$ (circle) (A3) and $u_d$ (arrow head) (B3); and the value of $\|\overline{g}(\varphi_\varepsilon)\|_{L^2(\Sigma)}$ against $\varepsilon$ with $u_d$ (circle) (A4) and $u_d$ (arrow head) (B4).}
		\label{figure:comp}
	\end{figure}
	
	To illustrate the effects of varying the value of $\varepsilon>0$ on the boundary $\Gamma(\varphi_\varepsilon)$, the values of the objective functional, the tracking functional and the $L^2(\Sigma)$-norm of $\overline{g}(\varphi_\varepsilon)$ we have \Cref{figure:comp}. We observe that for higher values of $\varepsilon$, the boundaries $\Gamma(\varphi_\varepsilon)$---as shown in \Cref{figure:comp} (A1) and (B1)--- converge towards $\Gamma(\circphi)$. This is due to the fact that higher values of $\varepsilon>0$ allow only small values of $\|\overline{g}(\varphi_\varepsilon)\|_{L^2(\Sigma)}$. In fact, \Cref{figure:comp} (A4) and (B4) show that the norm of $\overline{g}(\varphi_\varepsilon)$ tends to zero as $\varepsilon$ increases. On the other hand, we see from \Cref{figure:comp} (A1) and (B1) that as $\varepsilon$ decreases, the boundaries go farther from $\Gamma(\circphi)$. As expected from the result in \Cref{section:low-to-no}, we observe  convergence of the boundaries $\Gamma(\varphi_\varepsilon)$ as $\varepsilon\to 0$. Specifically, the boundaries are visually indistinguishable when $\varepsilon \le 0.125$ for the target profile $u_d$(circle) and when $\varepsilon \le 0.0625$ for $u_d$(arrow head). We also see that as $\varepsilon \to 0$, $\|\overline{g}(\varphi_\varepsilon)\|_{L^2(\Sigma)}$ increases and eventually  plateaus due to the box constraints imposed on the missing data. The behaviour of $\overline{g}(\varphi_\varepsilon)$ is also reflected in the behaviour of the tracking functional: as illustrated in \Cref{figure:comp} (A3) and (B3) it decreases as $\varepsilon$ increases, and increases up to a certain bound as $\varepsilon\to0$.

	For $\varepsilon\to 0$ the value of the tracking functional appears to plateau as a consequence of the fact that the constraints become active, while the value of the cost functional can still increase as reflected by the term on the right hand side of the above expression.  According to \eqref{noglo} it is bounded by the optimal value of the no-regret problem. 

	\subsection{Testing the regret}
	We now compare the low-regret solutions  to the solutions of shape optimization problem \eqref{objcon}--\eqref{poisson}  with specific values of the data $g_\delta\in \G$.
		In the examples that we shall show, we use the functions
		\begin{align*}
			g_\delta^i(x_1,x_2) = \min\left(0.2,\max\left(-0.2, \frac{x_1x_2}{30(1 - (0.099)i)} + 0.02\cos(4\pi x_1)\cos(4\pi x_2)\right)\right),
		\end{align*}
		for $i = 1,2,\ldots, 10$, and $x = (x_1,x_2)\in \Sigma$, which is designed to converge to a bang-bang type data as we increase the index $i$. 
		
		We use the notation $J_\delta^i(\cdot) = J(\cdot, g_\delta^i)$, and denote by $\varphi_\delta^i$ the minimizer of $J_\delta^i$, and test the case where the target profile is $u_d$ (bullet) since similar observations can be inferred from the other target profile. 
		We see from \Cref{figure:robtestbull} (A) that $J_\delta^i(\varphi_\delta^i) \le J_\delta^i(\varphi_\varepsilon)$, as expected. We also notice that as the index $i$ increases, the evaluations $J_\delta^i(\varphi_\delta^i)$ and $J_\delta^i(\varphi_\varepsilon)$ increase. This is attributed to the fact as the index gets higher the Dirichlet data satisfied by the state equation gets farther from the reference Dirichlet profile $g_r$.

		Another inference we can get from \Cref{figure:robtestbull} (A) is that the evaluations $J_\delta^i(\varphi_\varepsilon)$ do not depart  too far from the optimal value $J_\delta^i(\varphi_\delta^i)$. To see this clearly, and additionally to understand the effect of $\varepsilon$, we plotted the difference $J_\delta^i(\varphi_\varepsilon) - J_\delta^i(\varphi_\delta^i)$ against the parameter $\varepsilon$ in \Cref{figure:robtestbull} (B). We notice first that the difference, for any $i$ and $\varepsilon$, does not exceed $11\times 10^{-3}$. We can thus agree that the solutions $\varphi_\varepsilon$ are indeed of `low-regret'.

		In the same figures, we can also reflect on  behaviour of  $J_\delta^i(\varphi_\varepsilon) - J_\delta^i(\varphi_\delta^i)$ as the regularizing parameter $\varepsilon$ is varied. Decreasing  this parameter from $\varepsilon= 8 $ towards 0, the value of the difference first  gets smaller, before it starts to  increases again at about $\varepsilon = 0.225$.  
		As an attempt to explain this behavior we can expect that for large values of $\varepsilon$ the effect of the regularisation term $\frac{\varepsilon}{2} \|g_\delta\|^2_{L^2(\Sigma)}$ is dominating. The difference decreases as $\varepsilon$ is decreased until an optimal balance between regularisation and the $\min-\sup$  term in the low regret formulation \eqref{eqk1} is reached. Decreasing $\varepsilon$ even further the possible ill-posedness of the no-regret formulation shows effect. In our case it is dampened as a consequence of the constraints on $g_\delta$. This effect of the parameter $\varepsilon$ on the regret formulations is quite comparable to the one of the regularisation parameter in ill-posed inverse problems. This has led to a vast literature on optimal parameter-choice strategies  in the context of inverse problems, see e.g \cite{engl1996} and the literature cited there. It might  also be of interest to carry out investigations for the optimal choice of $\varepsilon$ for low-regret problems. We finally underline that in our test problem the  no-regret solution (obtained  as $\varepsilon \to 0$) provides  a deformation field leading to a respectable boundary $\Gamma(\varphi_\varepsilon)$, as seen from the fact that  the highest difference to the optimal cost does not  exceed $11\times 10^{-3}$.

		\begin{figure}[h]
			\centering
			\includegraphics[width=0.8\textwidth]{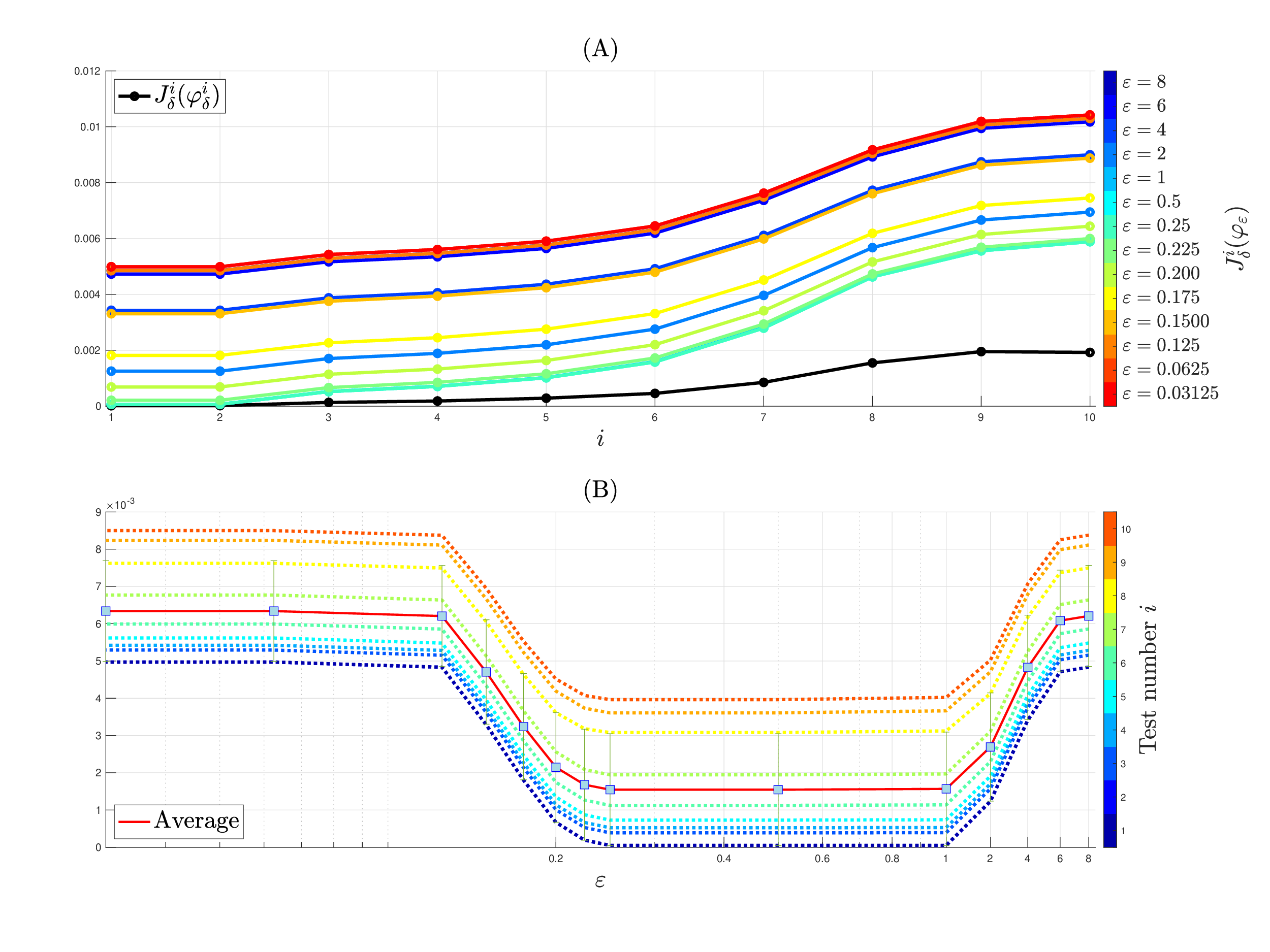}
			\caption{With target profile $u_d$(arrow head): comparison of $J_\delta^i(\varphi_\delta^i)$ with the evaluations $J_\delta^i(\varphi_\varepsilon)$ with the index $i$ on the $x$-axis and color coded according to $\varepsilon>0$ (A); plot of the differences $J_\delta^i(\varphi_\varepsilon)-J_\delta^i(\varphi_\delta^i)$ with $\varepsilon$ on the $x$-axis and color coded according to the index $i$ (B).}
			\label{figure:robtestbull}
		\end{figure}

		\section{Conclusion and outlook}\label{section:conclusion}

		A tracking-type shape optimization problem with missing Dirichlet boundary data $g_r$ was investigated. To address the challenge of missing data,  the optimization problem was reformulated into low-regret and no-regret problems. The low-regret problem is a regularized version of the no-regret problem. This reformulation introduced a new governing state equation with a one-way coupling, which we proved to be well-posed. Furthermore, we demonstrated that both the low-regret and no-regret problems admit solutions by employing convexity arguments. While the solution to the no-regret problem is not necessarily unique, we established that solutions to the low-regret problem converge to a solution of the no-regret problem as the regularization parameter $\varepsilon$ tends to zero. 
		
		Additionally,  a sensitivity analysis is provided, deriving the G{\^a}teaux derivative of the objective function for the low-regret problem. This allowed to design a gradient-based method to numerically solve the optimization problem. In our numerical examples, we demonstrated the convergence of the low-regret solutions as $\varepsilon \to 0$. Moreover, we showed convergence as $\varepsilon \to +\infty$, which corresponds to the case where the missing data is zero.
		
		We also explored how the solutions to the low-regret problem compare to those of the original optimization problem for specific choices of the data $g_\delta$. The results are promising, as the objective functional, evaluated at the solutions of the low-regret problem, remains close to the optimal value. Furthermore, we observed a threshold behavior: the difference between the objective values starts increasing  when moving toward smaller values from  $\varepsilon \approx 0.225$ on, and it also increases as $\varepsilon$ grows larger.

		In future work it can be of interest to investigate the concept of low-regret problems in context of inverse problems --{ which includes problems with measurements only on the surface }-- with missing data and to analyze the asymptotics as the regret parameter $\varepsilon$ and the noise level of the data tend to zero. {Another line of investigation could be the derivation of the optimality system for the no-regret problem. This includes analyzing the derivative of the objective functional for the low-regret problem as $\varepsilon\to 0$. }
		
		{
		We also note that if the non-convexity of the arrow head $\Gamma_d$ becomes more pronounced the accuracy of the low-regret approach will suffer. This can be considered to be a limitation of the method. It is thus a challenge for future work to propose methods which simultaneously take into consideration missing data and the reconstruction of objects with a complicated geometry.}
		

		\vspace{.5in}
		\begin{center}
			\textbf{Karl Kunisch}\\
			Johann Radon Institute for Computational and Applied Mathematics\\
			Austrian Academy of Sciences\\
			Altenberger Strasse 69, 4040 Linz, Austria
			
			\medskip
			and 
			
			\medskip
			Institute of Mathematics and Scientific Computing\\
			University of Graz\\
			Heinrichstrasse 36, A-8010 Graz, Austria\\
			
			\medskip
			E-mail: \texttt{karl.kunisch@uni-graz.at}
		\end{center}

		\vspace{.5in}
		\begin{center}
			\textbf{John Sebastian H. Simon}\\
			Mathematisches Institut\\
			Universit{\"a}t Koblenz\\
			Universit{\"a}tsstrasse 1, 56070 Koblenz, Germany\\
			E-mail: \texttt{jhsimon@uni-koblenz.de; jhsimon1729@gmail.com}
		\end{center}
		
	\end{document}